\documentclass[english]{article}

\expandafter\let\csname ver@natbib.sty\endcsname\relax

\usepackage{epstopdf}
\usepackage[font=footnotesize]{caption}
\usepackage{subcaption}
\usepackage{graphicx} 
\usepackage{algorithm}
\usepackage[noend]{algpseudocode}
\usepackage{hyperref}
\usepackage[backend=bibtex,sorting=none,firstinits=true, abbreviate=true]{biblatex}
\makeatletter
\def\BState{\State\hskip-\ALG@thistlm}
\makeatother

\usepackage{amsthm,amsmath, amssymb}
\newtheorem{theorem}{Theorem}[section]

\newtheorem{lemma}[theorem]{Lemma}

\newtheorem{corollary}[theorem]{Corollary}

\usepackage{geometry} 
\begin{document}

\markboth{Edgar Dobriban}
{Efficient Computation of Limit Spectra}

%
%


\title{\large \sc Efficient Computation of Limit Spectra of Sample Covariance Matrices} \author{ \normalsize Edgar Dobriban\thanks{Department of Statistics, Stanford University, Stanford, CA, USA. e-mail: \href{mailto:dobriban@stanford.edu}{dobriban@stanford.edu}}}
\maketitle


\begin{abstract}
Consider an $n \times p$ data matrix $X$ whose rows are independently sampled from a population with covariance $\Sigma$. When $n,p$ are both large, the eigenvalues of the sample covariance matrix are substantially different from those of the true covariance. Asymptotically, as $n,p \to \infty$ with $p/n \to \gamma$, there is a deterministic mapping from the population spectral distribution (PSD) to the empirical spectral distribution (ESD) of the eigenvalues. The mapping is characterized by a fixed-point equation for the Stieltjes transform. 

We propose a new method to compute numerically the output ESD from an arbitrary input PSD. Our method, called \textsc{Spectrode}, finds the support and the density of the ESD to high precision; we prove this for finite discrete distributions. In computational experiments it outperforms existing methods by several orders of magnitude in speed and accuracy.  We apply  \textsc{Spectrode} to compute expectations and contour integrals of the ESD. These quantities are often central in applications of random matrix theory (RMT).

We illustrate that \textsc{Spectrode} is directly useful in statistical problems, such as estimation and hypothesis testing for covariance matrices. Our proposal may make it more convenient to use asymptotic RMT in aspects of high-dimensional data analysis.
\end{abstract}


\section{Introduction}

Large data matrices are now commonly analyzed in science and engineering. Models from random matrix theory (RMT) are becoming increasingly used to understand the behavior of popular statistical methods on such matrices. RMT is particularly applicable to analyze statistical methods which depend on the sample covariance matrix of the data: for instance principal component analysis (PCA), classification, hypothesis testing of high-dimensional means, and independence tests, see e.g. the monographs  \cite{tulino2004random, serdobolskii2007multiparametric, couillet2011random, yao2015large}. 

Concretely, consider an $n \times p$ matrix $\mathbf{X}$, whose rows ${\bf x}_i$ are independent and identically distributed random vectors. Suppose that ${\bf x}_i$ are mean zero, and their covariance matrix is the $p \times p$ matrix ${\bf\Sigma} = \mathbb{E}{{\bf x}_i {\bf x}^\top_i}$. To estimate ${\bf\Sigma}$, we form the sample covariance matrix $\widehat{{\bf\Sigma}} = n^{-1} \mathbf{X}^\top \mathbf{X}.$ In the asymptotic model classically used in statistics, when $p$ is fixed and $n \to \infty$, the sample covariance matrix is a good estimator of the population covariance \cite{anderson1958introduction}. 

However, if $n$ and $p$ are of comparable size, then $\widehat{{\bf\Sigma}}$ deviates substantially from the true covariance. The asymptotic theory of random matrices describes the behavior of the eigenvalues of $\widehat{{\bf\Sigma}}$ as $n,p$ grow large proportionally, see \cite{bai2009spectral}. If the distribution of the eigenvalues of ${\bf\Sigma}$ tends to a limit population spectral distribution (PSD) as $n,p \to \infty$ and the aspect ratio $p/n \to \gamma$, then under mild conditions the random eigenvalue distribution of $\widehat{{\bf\Sigma}}$ also tends to a deterministic limit empirical spectral distribution (ESD) \cite{marchenko1967distribution,silverstein1995strong}.

The ``fundamental theorem of applied statistics'' is the statement that often the limit of the empirical distribution function is the population distribution. This theorem applies in numerous settings, see e.g. \cite{serfling2009approximation}, but not here. When $n \to \infty$ but $p/n \to \gamma>0$, the limit empirical spectrum differs from the true spectrum, because the number of samples is only a constant multiple of the dimension. This is very different from the case where $p$ is fixed and $n \to \infty$, in which case the sample spectrum converges to the true spectrum. The difference between the population and empirical eigenvalues has fundamental implications for high-dimensional statistical inference, see e.g. \cite{johnstone2007high, yao2015large}.  It becomes central to high-dimensional statistical analysis to understand the relationship between the population and sample eigenvalues. This understanding should help adjust classical statistical methods to the high-dimensional setting.

However, the relationship between the population and sample spectrum is complex, implicit and non-linear; it is described by a fixed-point equation - often called the Marchenko-Pastur equation or Silverstein equation - for the Stieltjes transform of the limit ESD. As a consequence, the ESD is not available in closed form, except for very special cases. The implicit description of the sample spectrum can be somewhat hard to understand, as well as hard to use in any practical setting, including data analysis. 

Reliable, precise, and efficient computational methods are needed  to understand the relationship between the population and sample eigenvalues. Perhaps surprisingly, research focused on delivering robust software tools for numerically computing large classes of limit ESDs has received relatively little attention. While there are important contributions to related problems (see Section \ref{relatedwork}), none of them are fully suitable for our problem.

The main method for computing the limit ESD is a fixed-point algorithm (\textsc{FPA}) which directly iterates the Silverstein equation. Since the algorithm is immediately suggested by the fixed-point characterization of the ESD, the history of this algorithmic approach is apparently lost in the prehistory of the subject. \textsc{FPA} has appeared recently in various forms, e.g. \cite{belinschi2007new, hachem2007deterministic, couillet2011deterministic,yao2012note,hendrikse2013smooth}. Further, \textsc{FPA} is recommended as the default method for computing the ESD in the monograph \cite{yao2015large}. \textsc{FPA} is a good method for computing the density of the ESD at a single point. However, usually the density of the ESD must be computed on a dense grid on the real line. When this is the case, we show that \textsc{FPA} is inefficient for high-precision computations. 

We propose the new method \textsc{Spectrode} to compute the limit empirical spectrum  of covariance matrices from the limit population spectrum.  \textsc{Spectrode} improves on \textsc{FPA} by exploiting the smoothness of the ESD.  We show in computational experiments on dense grids that our new method is dramatically faster and more accurate than FPA and other methods. For instance, on natural test problems in Section \ref{timing} below,  \textsc{Spectrode} is up to 1000 times faster than \textsc{FPA} while achieving the same accuracy! Finally, for atomic PSDs, i.e. weighted mixtures of point masses, we prove its convergence to the correct answer. 

\textsc{Spectrode} is publicly available in an open source \textsc{Matlab} implementation at \url{https://github.com/dobriban/eigenedge}. This software package also has the code to reproduce all computational results of our paper (see Section \ref{software}).

In the remainder of the introduction, we showcase example computations with \textsc{Spectrode}, and highlight the key aspects of the method. Then we state its properties more precisely.

\subsection{Two examples}

\begin{figure}
\centering
\begin{minipage}{.5\textwidth}
  \centering
  \includegraphics[scale=0.33]
  {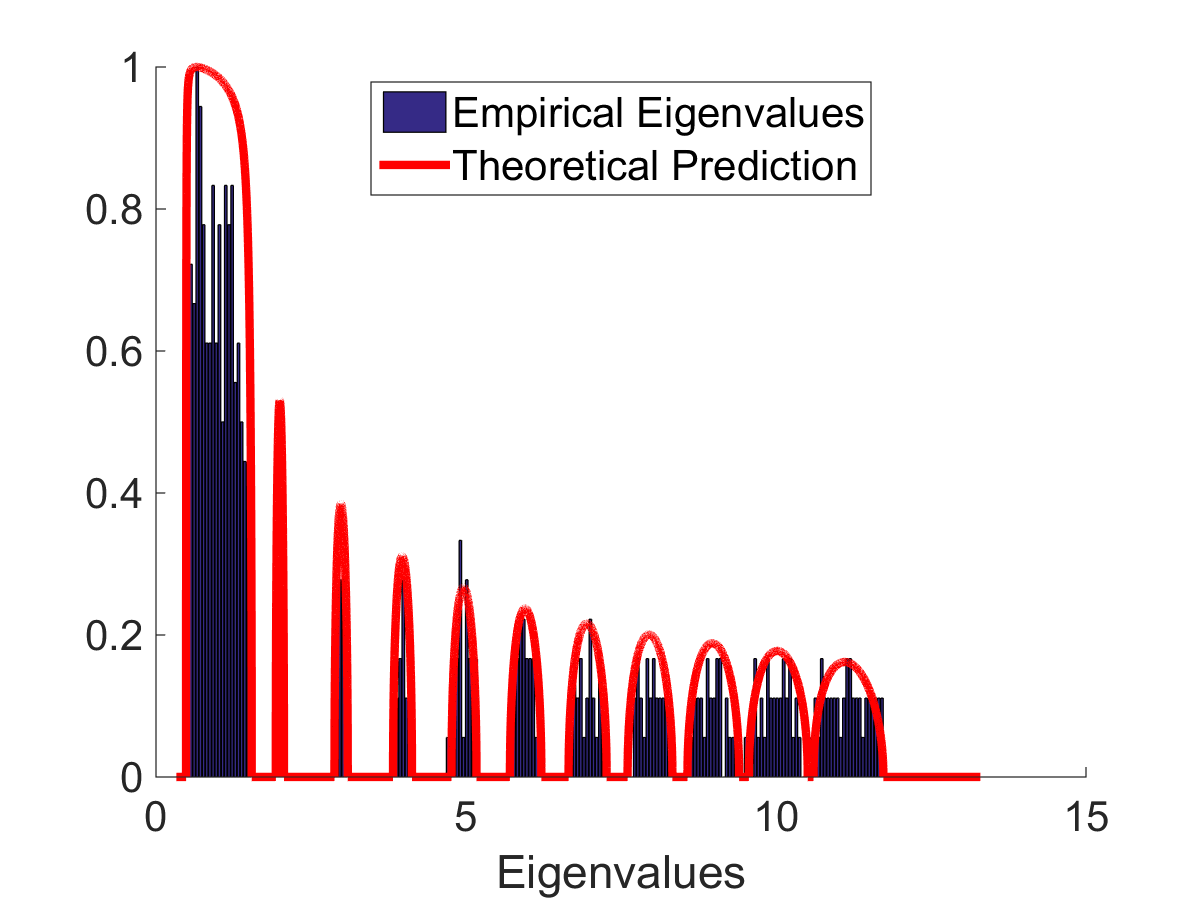} \\{\footnotesize (a) Density of limit ESD.}
\end{minipage}%
\begin{minipage}{.5\textwidth}
  \centering
  \includegraphics[scale=0.33]{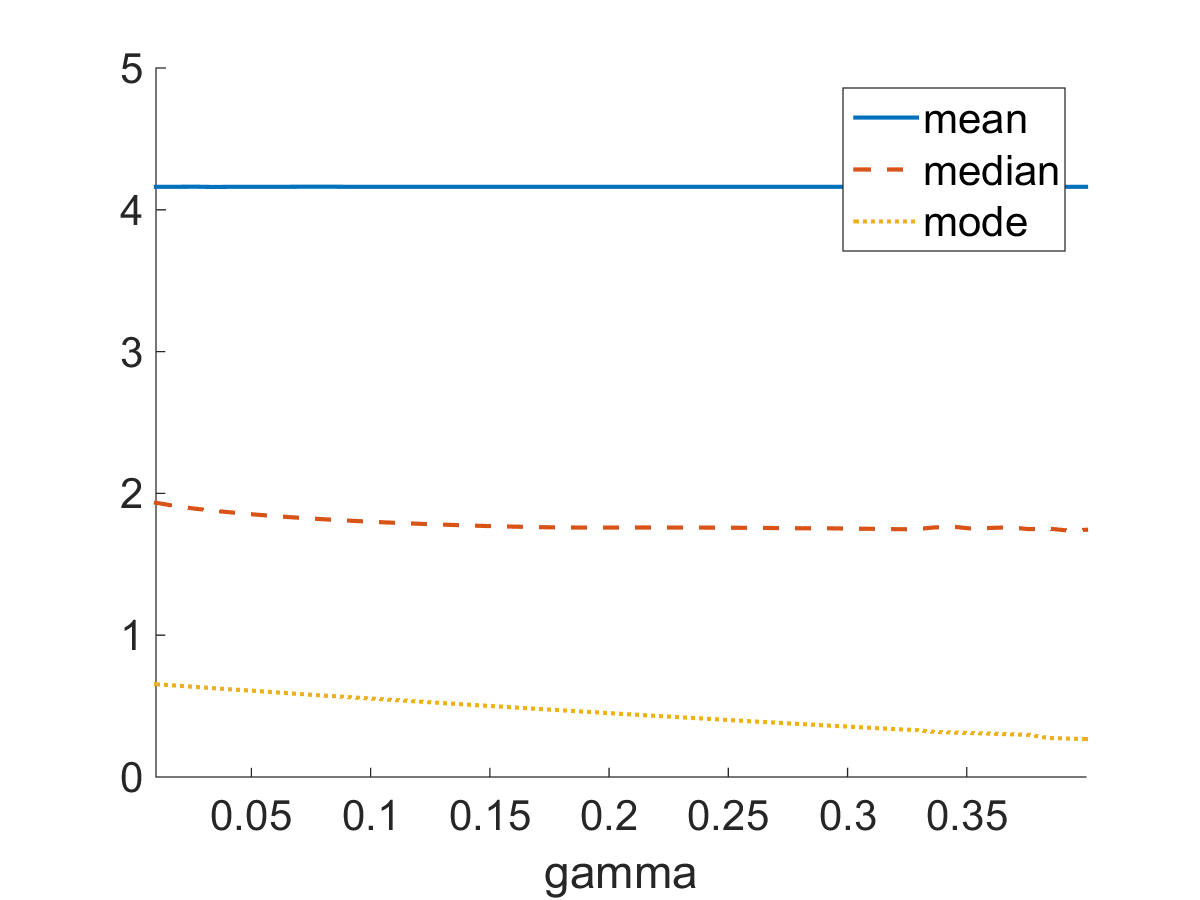}
 \\{\footnotesize (b) Mean, median, and mode of the limit ESD.}
\end{minipage}
\caption{Boxcar + point mass mixture example: \textsc{Spectrode} computes  (a) the density of the limit ESD, normalized to have maximum equal to one for display purposes, and  (b) its mean, median, and mode as a function of $\gamma$.}
\label{two_examples}
\end{figure}

We illustrate \textsc{Spectrode} by computing two spectral quantities of interest. In this example the PSD is an equal mixture of two components: (1) a mixture of ten point masses at $2$,$3,\ldots,11$, with weights forming an arithmetic progression with step $r = 0.005$ as follows: $0.0275, 0.0325,\ldots, 0.0725$; and (2) a uniform distribution - or a `boxcar' - on $[0.5,1.5]$, with mixture weight 1/2. The weights sum to one. We use the aspect ratio $\gamma = 0.01$. 

Figure \ref{two_examples} shows the example computations. In subplot (a), we show the key output of \textsc{Spectrode}, the density of the limit ESD. The computation takes 0.5 seconds on a desktop computer. \emph{A priori} it is not obvious how many disjoint clusters there are in the ESD, or what their shape is. Several insights can be derived from the computation: there are 11 clusters in total, so all population  clusters separate. Each population cluster in the PSD, in this instance, creates a distinct component of the ESD. Further, the two rightmost clusters nearly touch; and the height of the clusters decreases while the width of the point-mass clusters increases. We are not aware of any other method besides \textsc{Spectrode} from which these properties can be derived with comparable speed.

As a second example, in subplot (b) we compute three functionals of the ESD as a function of $\gamma$: the mean, median, and the mode. Such functionals are important in statistical applications: for instance, the median is used for optimal singular value shrinkage in \cite{gavish2014optimal}; see also Section \ref{moments}.  As expected, the mean does not depend on $\gamma$. However, the behavior of the median and the mode is not obvious. Using \textsc{Spectrode}, one can get insight into their behavior: the mode decreases as a function of $\gamma$, and the median is greater than the mode. 

\subsection{Highlights}
To summarize and expand on the above argument, we highlight the following aspects of our method. 
\begin{enumerate}
\item It provides ready access to a wide variety of new examples of limit spectra of covariance matrices. 
There has been, until now, no convenient tool for precisely calculating the ESD for such large collections of examples.

\item \textsc{Spectrode} computes to high precision several important functionals of the limit empirical spectrum, namely: 
\begin{enumerate}
	\item{\bf The edges of the support:} The edges are of substantial interest in the study of phase transitions in spiked covariance models, for instance in \cite{benaych2012singular}, and in designing optimal singular value shrinkers for matrix denoising, for instance in \cite{Nadakuditi2014optshrink}. 

\item {\bf Moments of the ESD:} We show that general moments $\mathbb{E}_F [h(\lambda)]$ can be computed conveniently with \textsc{Spectrode}. The polynomial moments $\mathbb{E}_F [\lambda^k]$ can be computed alternatively via challenging free probability calculations, see \cite{nica2006lectures}. However, this does not hold for more general moments $\mathbb{E}_F [h(\lambda)]$ for arbitrary $h$. Therefore, our method could simplify and extend the applicability of existing techniques by providing a unified way to compute nearly all global spectral moments of interest (Section \ref{moments}).

\item {\bf Contour integrals of the ESD's Stieltjes transform:} Contour integrals of the Stieltjes transform appear crucially in the central limit theorem for linear spectral statistics (LSS) of covariance matrices due to Bai and Silverstein \cite{bai2004clt}.  In applications of this powerful result to multivariate statistics, calculating the contour integral formulas for the mean and the variance is a key step, see e.g. \cite{yao2015large}. These moments are known in closed form only in a few cases.  The current approach is to calculate them using residue theory from complex analysis. This analytic approach can require substantial effort, and is limited to the cases where the ESD is known in closed form. \textsc{Spectrode} enables us to compute such contour integrals numerically instead (Section \ref{contour}). High precision numerical results may suffice in many applications.
\end{enumerate} 

\item \textsc{Spectrode} is directly useful in statistical applications.  We give two key examples where \textsc{Spectrode} could, in our view, significantly improve on current statistical methodology:

\begin{enumerate}
\item{\bf Estimating the covariance matrix ${\bf\Sigma}$:} A problem of considerable interest in statistics is to estimate the unobserved covariance matrix ${\bf \Sigma}$ based on the observed data. When the number of samples $n$ is comparable to the dimension $p$, covariance estimation is a challenging problem. 

A recent series of methods due to Ledoit and Wolf \cite{ledoit2013optimal,LW_spectrum} assumes that one can accurately compute the ESD for any proposal PSD. It repeatedly invokes this ability as the `engine' driving its core iteration. Unfortunately, the whole framework is limited by the accuracy of the ESD computation. \textsc{Spectrode} allows to immediately upgrade this procedure, replacing the existing low-precision ESD computations with high-precision ones.

\item {\bf Hypothesis tests on the covariance matrix:} Testing statistical hypotheses on the covariance matrix can be approached by using the CLT for the linear spectral statistics, see e.g. \cite{bai2009corrections,yao2015large}. As discussed above, we suggest here that the mean and variance in the CLT could be computed numerically (Section \ref{contour}). Our approach, implemented with open source software, might be significantly more convenient than traditional analytic calculations, compare \cite{bai2009corrections,yao2015large}. In addition, it may lead to entirely new test statistics, whose analysis was not possible via pre-existing methodology.
\end{enumerate}

\end{enumerate}

There may be of course many other ways that our efficient computational framework will be useful to the statistics and engineering communities.

\subsection{Properties of \sc{Spectrode}}

\subsubsection{Background}

To state precisely the properties enjoyed by \textsc{Spectrode}, we first set up the formal background. A more thorough presentation will be given in Section \ref{themethod} below.  Consider a sequence of problems indexed by $p$, with deterministic $p \times p$ covariance matrices ${\bf\Sigma}_p$. Let $H_p$ be the distribution of eigenvalues of ${\bf\Sigma}_p$, i.e. $\tau_1,\ldots,\tau_p$ be the eigenvalues of ${\bf\Sigma}_p$, and $H_p$ the discrete distribution with cumulative distribution function $H_p(x) = p^{-1}\sum_i \mathrm{I}(\tau_i \le x)$. For each $p$, draw $n_p$ independent samples ${\bf x}_{ip} $ from a distribution whose covariance matrix is ${\bf\Sigma}_p$. The samples are of the form ${ \bf x}_{ip} = {\bf\Sigma}_p^{1/2}{\bf y}_{ip}$, where ${\bf y}_{ip}$ is a $p$-dimensional random vector with independent and identically distributed, mean zero, variance one entries. 

Arrange the vectors ${ \bf x}_{ip}$ into the rows of the $n \times p$ data matrix $\mathbf{X}_p$. Form the sample covariance matrix $\widehat{{\bf\Sigma}}_p = {n_p}^{-1} \mathbf{X}_p^\top \mathbf{X}_p$. Let $F_p$ be the distribution of the $p$ eigenvalues of $\widehat{{\bf\Sigma}}_p$: thus $\lambda_1,\ldots,\lambda_p$ are the eigenvalues of $\widehat{{\bf\Sigma}}_p$, and $F_p$ is discrete distribution with cumulative distribution function $F_p(x) = p^{-1}\sum_i \mathrm{I}(\lambda_i \le x)$.

Consider the high-dimensional limit where $n,p \to \infty$ such that $p/n_p \to \gamma$. Suppose the eigenvalue distributions $H_p$ converge to a limit population spectral distribution (PSD) $H$, i.e. $H_p \Rightarrow H$ in distribution. Then a cornerstone result in random matrix theory, the Marchenko-Pastur theorem, states that the empirical eigenvalue distributions $F_p$ also converge, almost surely, to a limit empirical spectral distribution (ESD) $F$ \cite{marchenko1967distribution,silverstein1995strong}.

We consider the computation of $F$ from $H$. The method we propose is general and well-defined for all population spectral distributions $H$. Our analysis considers \emph{atomic} PSDs $H$, which are finite mixtures of point masses, but see Section \ref{non_atomic} for the extension to general distributions. Thus we assume $H = \sum_{i=1}^{J} w_i \delta_{t_i}$.

where $\delta_t$ is the point mass at $t$, $w_i > 0$ are the component masses with $\sum_i w_i = 1$, and $t_i > 0$ are the corresponding population eigenvalues. We exclude the case $\gamma = 1$ for technical reasons, specifically the potentially unbounded density of the ESD at $x=0$. 

In pioneering work, Silverstein and Choi \cite{silverstein1995analysis} study the limit ESD corresponding to general $H$ in depth. They show that the limit ESD $F$ has a continuous density $f(x)$ for $x\neq 0$. The density $f(x)$ exists at $x=0$ if $\gamma <1$, but not if $\gamma>1$. Instead $F$ has a point mass of weight $1-\gamma^{-1}$ at $x=0$.  For atomic distributions, it follows from the results in \cite{silverstein1995analysis} that the distribution is supported on a union of $K$ disjoint compact intervals $[l_k,u_k]$, where $l_k$ is the lower endpoint and $u_k$ is the upper endpoint of the $k$-th interval for $1 \le k \le K$. The endpoints are such that $0 \le l_1 < u_1 < \ldots < l_K < u_K$.  The number of sample intervals $K$ is at most the number of population components $J$. If the aspect ratio $\gamma = \lim p/n $ is sufficiently close to 1, then  some population components can ``merge'' in the sample spectrum, and $K <J$ will occur. Finally, it is shown in \cite{silverstein1995analysis} that $f$ is analytic in the neighborhood of all points where the density is positive.

\subsubsection{Input and output of \sc{Spectrode}} Given the aspect ratio $\gamma$, a population spectrum $H$ (for instance an atomic distribution) and a user-specified precision control parameter $\varepsilon >0$, \textsc{Spectrode} produces numerical approximation of $F$ consisting of:
\begin{enumerate}
\item The number of intervals in the support of $F$: $\hat{K} = \hat{K}(\varepsilon)$. 
\item The endpoints of the support intervals $[\hat l_k(\varepsilon),\hat u_k(\varepsilon)]$, for $k = 1, \ldots,  \hat{K}$.
\item The density $\hat{f}(x, \varepsilon)$ for all real $x$.
\end{enumerate}

For the reader's convenience, the input and output of \textsc{Spectrode} is summarized in Table~\ref{black_box}.

\begin{table}[h]
\footnotesize
\centering
\caption{Input and Output of {\sc Spectrode}}\label{black_box}
\begin{tabular}{|l|}
\hline
{\sc Spectrode}: Input and Output                                                                                \\ \hline
{\bf Input: }                                                                                                      \\ \mbox{     }$H \gets$ population spectrum\\ 
\mbox{     }\mbox{     }\mbox{     }\mbox{     }(e.g. atomic measure: eigenvalues $t_1, \ldots, t_J$ and masses $w_1, \ldots, w_J$) \\
\mbox{     }$\textit{$\gamma$} \gets \text{ aspect ratio}$                                                                    \\
\mbox{     }$\varepsilon \gets \text{precision control parameter}$                                                            \\ \hline
{\bf Output:}                                                                                                 \\ 
\mbox{     }$\hat{K}(\varepsilon) \gets $ number of intervals in the support                                                  \\
\mbox{     }$[\hat l_k(\varepsilon),\hat u_k(\varepsilon)] \gets $ endpoints of intervals in the support                      \\
\mbox{     }$\hat{f}(x, \varepsilon)\gets $ density of the spectrum, for any $x$                                              \\ \hline
\end{tabular}
\end{table}

\subsubsection{Correctness of \sc{Spectrode}} 

Our main theoretical results, given in Section \ref{themethod} below, demonstrate the correctness of our proposed method. As the user-specified precision control parameter $\varepsilon \to 0$, \textsc{Spectrode} has the following performance characteristics:

\begin{enumerate}
\item Correctness of the \textit{number of disjoint intervals} of the support:
\begin{equation}
\lim _{\varepsilon \to 0} \hat{K}(\varepsilon) = K. 
\label{recover_num_disjoint}
\end{equation}
\item Accuracy of the \textit{endpoints} of the support:
\begin{equation}
\lim _{\varepsilon \to 0} \hat l_k(\varepsilon) = l_k, \mbox{ and }  \lim _{\varepsilon \to 0} \hat u_k(\varepsilon) = u_k. 
\label{compute_endpoints}
\end{equation}
\item Accuracy of the \textit{density}\footnote{A more precise statement is: For $\gamma<1$, the convergence is uniform over all $x\in \mathbb{R}$. For $\gamma>1$, the density does not exist at $x=0$, but is equal to $f(x)=0$ on some intervals $\mathcal{I} = (-\delta,0) \cup (0,\delta)$, with $\delta>0$. Then, the convergence is uniform over the closed set $ \mathbb{R}\setminus\mathcal{I}$.}:
\begin{equation}
\lim _{\varepsilon \to 0} \sup_{x \in \mathbb{R}\setminus\{0\}}|\hat f(x,\varepsilon)- f(x)| = 0. 
\label{compute_density}
\end{equation}
\end{enumerate}

Claims \eqref{recover_num_disjoint}-\eqref{compute_endpoints} are proved in Theorem \ref{endpoints_accurate}, while claim \eqref{compute_density} is proved in Theorem \ref{density_approx}. The claims are verified in reproducible computational experiments in the next section (see also Section \ref{software}). 

As a consequence of these results, we show in Corollary \ref{moments_approx} that the moments of the limit ESD can be accurately computed by integrals against the approximated density. Finally, we adapt \textsc{Spectrode} to compute contour integrals involving the Stieltjes transform of the limit ESD in Section \ref{contour}. 

The computational framework used by \textsc{Spectrode} is applicable to general population distributions $H$, not just atomic distributions. Indeed, we already showed an example involving a uniform distribution in Figure \ref{two_examples}. However, our current software implementation of \textsc{Spectrode} assumes that $H$ is a finite mixture of uniform distributions and point masses. Moreover, the proof of convergence that we supply in this paper only holds for atomic distributions. Therefore, we will work with atomic distributions through most of this paper. This issue is further discussed in Section \ref{non_atomic}.

In the rest of the paper, we validate our claims with computational experiments in Section \ref{numerical_results}. \textsc{Spectrode} and its convergence in presented in Section \ref{themethod}. After giving some applications and extensions in Section \ref{Consequences and Extensions}, we describe related literature in Section \ref{relatedwork}. The available software and the tools to reproduce our computational results are described in Section \ref{software}.

\section{Computational Results}
\label{numerical_results}

In addition to theoretical correctness results, we validate our performance claims from the introduction by computational experiments. We present supporting evidence  for claims \eqref{recover_num_disjoint} and \eqref{compute_endpoints} on the correctness of the support in Section \ref{exp:support}; for claim \eqref{compute_density} on the correctness of the density in Section \ref{sec:test_correctness}; and for the computational efficiency of \textsc{Spectrode} in Section \ref{timing}. The experiments are reproducible (see Section \ref{software}).

\subsection{Correctness of the support}
\label{exp:support}

\subsubsection{The comb model}
To show that our algorithm identifies the support (Claims \eqref{recover_num_disjoint} - \eqref{compute_endpoints}), we consider the following \emph{comb model} for eigenvalues. Here the eigenvalues and the weights are each defined in terms of arithmetic progressions
$ H = \sum_{j=0}^{J-1} (a + jb) \delta_{c+jd}$.

The eigenvalues are placed at $c+jd$, for some $c>0$ and $d \in \mathbb{R}$ such that $c+jd>0$ for all $j$. They have weights $a + jb$ for some $a,b>0$. The constants $a$, $b$ are constrained so that the sum of the weights is one, thus only one of them, say $b$, is a free parameter. This is a flexible model governed by only three parameters. 

The comb model is useful for gaining insight into the support identification problem. 
Interesting behavior occurs as a function of the problem parameters $J, b,c,d$ and $\gamma$. For instance, let us change $\gamma$ while fixing all other variables. If $\gamma \to 0$, then $F_{\gamma} \to H$ \cite{silverstein1995analysis}; intuitively the number of samples is much larger than the dimension, $n \gg p$, so the ESD converges to the atomic population SD. Now as $\gamma$ increases, the sharp atoms spread out into density bumps. 

If the original atoms are sufficiently close to each other, then at some point bumps will start merging. The precise moment when this happens is in general a complicated function of $J, b,c,d$ and $\gamma$, but can be determined precisely with \textsc{Spectrode}. Hence we provide a useful tool for understanding and exploring support identification. 

\subsubsection{Testing our method}  

Since in most cases there is no closed form for the density, we compare our algorithm against the Fixed-Point Algorithm (\textsc{FPA}); see its description preceding Lemma \ref{ode_unique}. \textsc{FPA} is empirically slow for dense grid evaluation (Section \ref{timing}), but converges, as shown in a more general setting in \cite{couillet2011deterministic}. 

Since the convergence rate is not known, one cannot guarantee the exact accuracy of \textsc{FPA}. We have validated \textsc{FPA} separately on simpler test cases where the closed form expression was known (data not shown).

More specifically, our numerical test has the following framework: For given problem parameters, and an accuracy control parameter $\varepsilon$, we run \textsc{Spectrode} to produce numerical approximations $\hat K(\varepsilon)$, $\hat l_k(\varepsilon)$,  $\hat r_k(\varepsilon)$. The method also returns a dense grid of $x_i$. On this grid we compute the density approximations $\hat{f}_{\mathrm{fp}}(x_{i}, \varepsilon_0)$ of FPA, with an accuracy control parameter $\varepsilon_0$. Here the parameter $\varepsilon_0$ is smaller than $\varepsilon$, so that the fixed-point algorithm's solution can be reliably used as a basis of comparison for $\varepsilon$-accurate computations. 

We then define the gold standard approximation to the support as the connected components of the grid $x_i$ where the density $\hat{f}_{\mathrm{fp}}(x_{i}, \varepsilon_0) > \varepsilon_0$. This step thresholds the density at level $\varepsilon_0$, because FPA was tuned to have accuracy of the order $\varepsilon_0$, its control parameter. This prescription produces approximations $\hat K_{\mathrm{fp}}(\varepsilon_0)$, $\hat l_{\mathrm{fp},k}(\varepsilon_0)$,  $\hat r_{\mathrm{fp},k}(\varepsilon_0)$, which we use to evaluate \textsc{Spectrode}.

We evaluate \textsc{Spectrode} by calculating the error in the number of clusters:  $\Delta_K(\varepsilon) = |\hat K(\varepsilon) - \hat K_{\mathrm{fp}}(\varepsilon_0)|$, where $\varepsilon_0$ is suppressed for brevity. For the support endpoints we proceed similarly. If the number of intervals is not computed correctly, then we set this error to $+\infty$. Even if $\hat K=K$, we have to take into account the finite precision of the grid $x_i$. 

Consider a lower endpoint for one of the clusters. Suppose the two methods return the grid elements $x_i$ and $x_j$, with $i\le j$, as numerical approximations for the lower endpoint. Due to finite grid precision, $|x_i-x_j|$ can be an understimate of the actual error. For instance if $x_i=x_j$, it's clear that our accuracy bound cannot, in general, be better than the size of grid spacings $|x_i-x_{i-1}|$, $|x_j-x_{j+1}|$. Generally, a conservative estimate of the accuracy can be obtained by adding these neighboring grid spacings to $|x_i-x_j|$ to get (recall $x_{i-1} < x_{i}\le x_{j} <x_{j+1}$): $\Delta^l_k(\varepsilon) = |x_{i-1}-x_{j+1}|$.

Finally the approximation error for lower endpoints is defined as the average of all errors for lower endpoints:
$\Delta_l(\varepsilon) = \sum_{k=1}^{\hat K}\Delta^l_k(\varepsilon)/\hat{K}$.
The approximation error for upper endpoints $\Delta_u(\varepsilon)$ is defined analogously. Note again: if $\hat K \neq K$, we set the error to be $\infty$.

\begin{figure}
\centering
\begin{minipage}{.5\textwidth}
  \centering
  \includegraphics[scale=0.33]
  {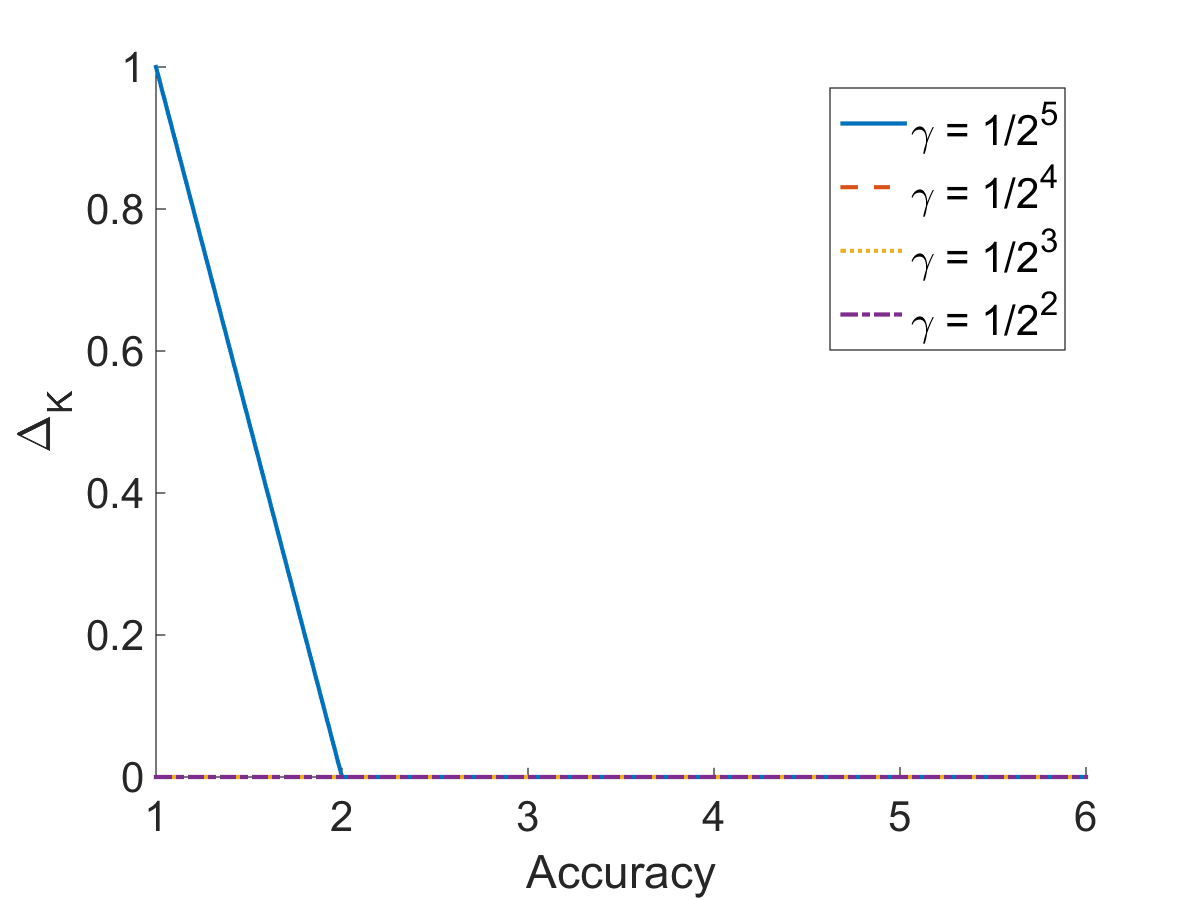}\\{\footnotesize $\Delta_K(\varepsilon)$: error in the number of clusters.}
\end{minipage}%
\begin{minipage}{.5\textwidth}
  \centering
  \includegraphics[scale=0.33]{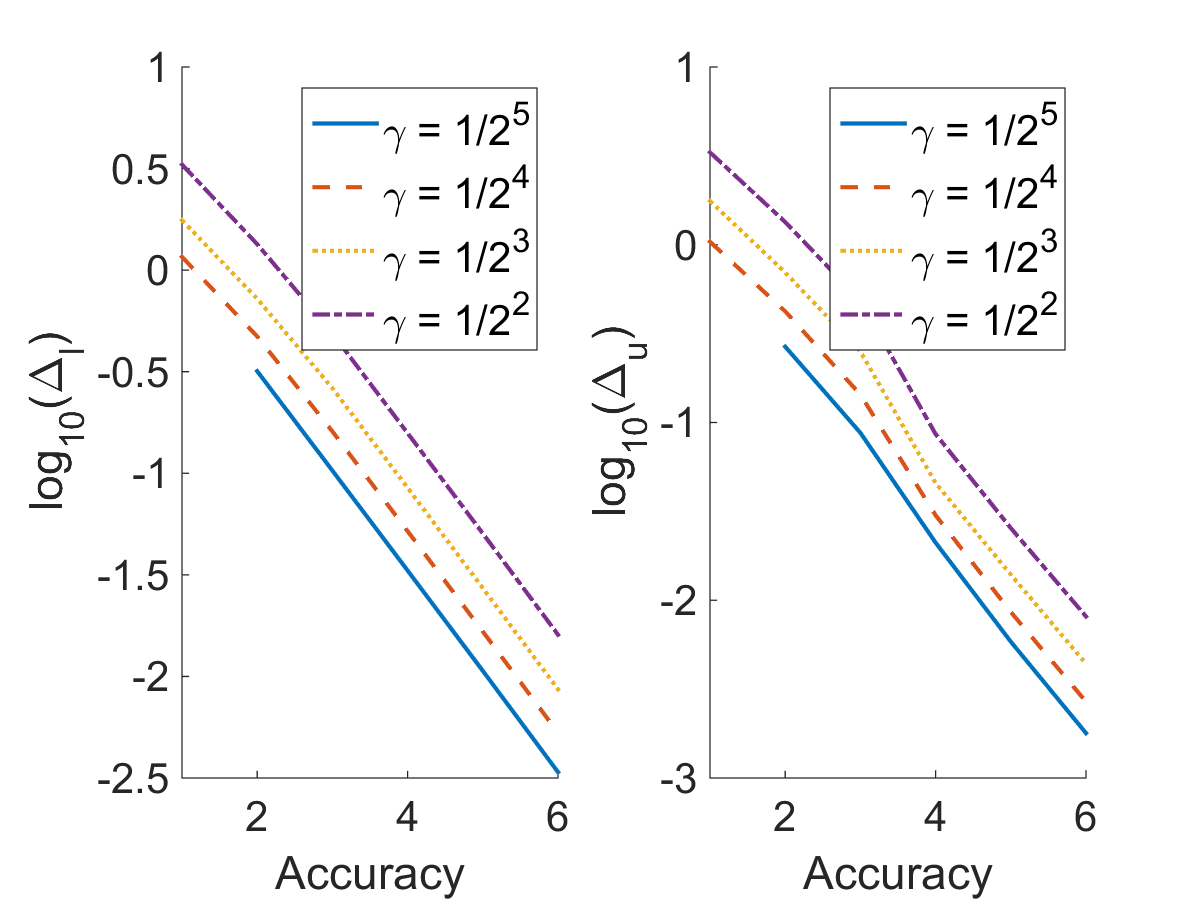}
  \\{\footnotesize $\Delta_l(\varepsilon),\Delta_u(\varepsilon)$: error in the cluster endpoints.}
\end{minipage}
\caption{\textsc{Spectrode} (a) correctly identifies the number of disjoint intervals in a comb model; and (b) accurately computes the lower and upper endpoints.}
\label{support_identification}
\end{figure}

The comb model in this test has $J=6$  clusters spaced evenly between $1/2$ and $10$, and a gap in the sequence of weights $b=0.01$, leading to nearly equal weights. The aspect ratio $\gamma$ takes fourf values between $1/2^5$ and $1/2^2$. We fix $\varepsilon_0 = 10^{-7}$ and vary the accuracy $\varepsilon = 10^{-m}$, $m=1, \ldots ,6$.

\subsubsection{Results}
We show the results of the experiment on Figure \ref{support_identification}. In panel (a), we show the error in the number of clusters $\Delta_K(\varepsilon)$ for the four different aspect ratios $\gamma$, as a function of the accuracy. \textsc{Spectrode} makes at most one error in the number of clusters. For sufficiently high accuracy the number of clusters is correct.

In panel (b), we show the approximation error for the endpoints $\Delta_l(\varepsilon)$ (left), and $\Delta_u(\varepsilon)$ (right), on a logarithmic scale. For the experiments where $\Delta_K(\varepsilon)>0$, we leave blanks. We observe that for all values of $\gamma$ the approximation gets better with higher accuracy. This convergence appears nearly linear in $\varepsilon$: number of correct digits is approximately linearly related to $-\log_{10}(\varepsilon)$, with a slope of approximately 1/2.  These experiments provide evidence for our claims \eqref{recover_num_disjoint} - \eqref{compute_endpoints}: \textsc{Spectrode} correctly identifies the support, as the precision parameter $\varepsilon \to 0$. 

\subsection{Accurate computation of the density}
\label{sec:test_correctness}

\begin{figure}
\centering
\begin{minipage}{.5\textwidth}
  \centering
  \includegraphics[scale=0.33]
  {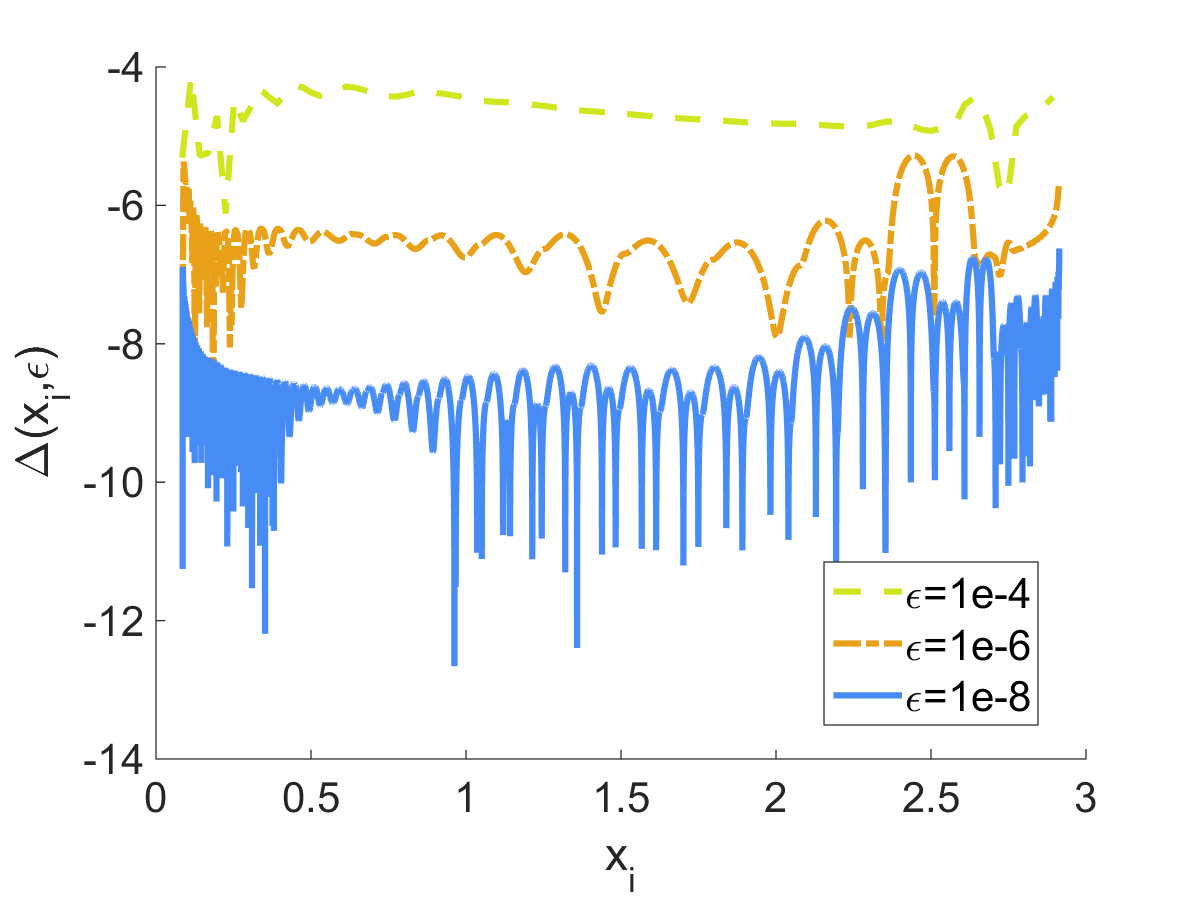}
  \\{\footnotesize (a)  {\tt MP}}
\end{minipage}%
\begin{minipage}{.5\textwidth}
  \centering
  \includegraphics[scale=0.33]{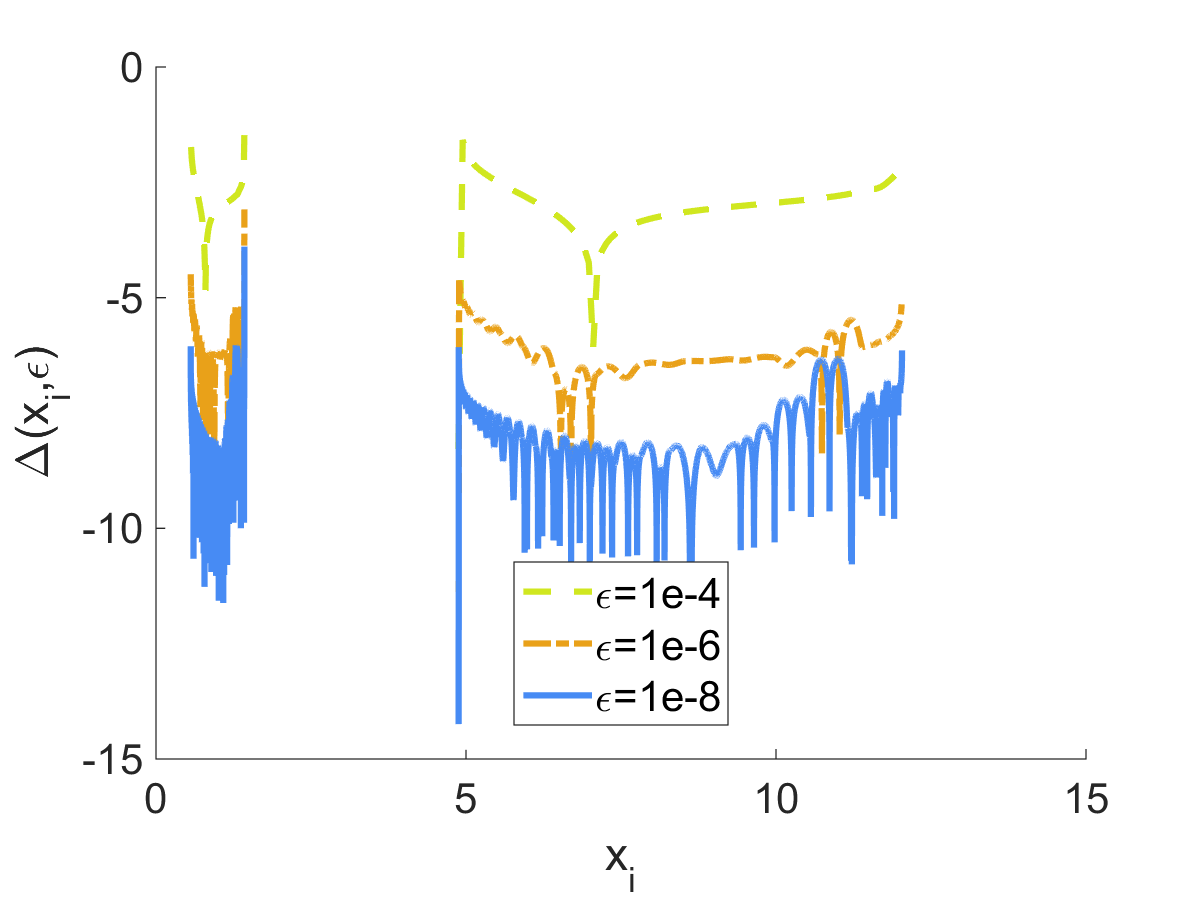}
  \\{\footnotesize (b)  {\tt TwoPoint}}
\end{minipage}
\caption{Accurate computation of the density (Section \ref{sec:test_correctness}) in two test problems. Left panel: {\tt MP}. Right panel: {\tt TwoPoint}. We display the the error in the density $\Delta(x_i,\varepsilon)$ for three different values of $\varepsilon$, $10^{-4},10^{-6},10^{-8}$.}
\label{test_correctness}
\end{figure}

We validate our claim \eqref{compute_density} that \textsc{Spectrode} accurately computes the limit density. To test the accuracy up to several digits, we rely on examples where the density $f$ can be found exactly in an alternative way. 

\subsubsection{Test problems}

For our first test, called {\tt MP}, the population spectrum $H$ is a point mass at 1. The ESD has the well-known density:

\begin{equation}
f(x; \gamma) = \frac{\sqrt{(\gamma_+- x)(x-\gamma_-)}}{2\gamma x}I(x\in[\gamma_-,\gamma_+]),
\label{MP_density}
\end{equation}

where $\gamma_\pm = (1 \pm \sqrt{\gamma})^2$. The distribution of eigenvalues has a point mass at $x=0$ if $\gamma >1$. 

For the second test, called {\tt TwoPoint}, $H$ is a mixture of two point masses at $x=1$ and $t$, with weights $q$ and $1-q$, then the Silverstein equation \eqref{silv.eq} for $v(z)$ becomes
$$  
-\frac{1}{v(z)} = z - \gamma\left(\frac{q}{1 + v(z)} + \frac{(1-q)t}{1 + tv(z)}\right).
$$
This is equivalent to a polynomial equation in $v$ of degree at most three:
\begin{equation}
zt v^3 + (zt + z + t - t\gamma)v^2 + \left[z + t + 1 - \gamma\left(q + (1 - q) t\right)\right)] v + 1 = 0. 
\label{cubic}
\end{equation}

When $z, t \neq 0$, as is always the case for us, this is a cubic equation in $v$, which can be solved exactly. The theory of Silverstein and Choi \cite{silverstein1995analysis} guarantees that for real $x$ within the spectrum support of the spectrum, Eq. \eqref{cubic} has exactly one root with positive imaginary part. This is guaranteed to lead to the correct density.  For real $x$ outside  the spectrum, Eq. \eqref{cubic} has three real roots. This distinguishes the inside from the outside.

For each grid point $x_i$ and accuracy $\varepsilon$, \textsc{Spectrode} produces a numerical approximation $\hat f(x_{i},\varepsilon)$ to the true density $f(x_{i})$. To test \textsc{Spectrode}, we compute the error in the density:

\begin{equation}
\label{error_in_density}
\Delta(x_i,\varepsilon) = \log_{10}|\hat f(x_{i},\varepsilon) - f(x_{i})|. 
\end{equation}

We set $\gamma=1/2$, and vary the global accuracy parameter $\varepsilon$ in powers of ten as $10^{-4},10^{-6},10^{-8}$. In addition, for the two-point mixture model we set a fraction $q = 1/2$ of the eigenvalues to $t = 8$.

\subsubsection{Results} 
The results are shown in Figure \ref{test_correctness}. For both test problems, the error in the density decreases uniformly as the tuning parameter $\varepsilon \to 0$. Furthermore, \textsc{Spectrode} produces approximately the required accuracy: for instance the average precision for $\varepsilon = 10^{-8}$ is approximately eight digits. These experiments provide empirical evidence for claim  \eqref{compute_density}: \textsc{Spectrode} computes the density of the limit ESD with uniform accuracy over all $x$.

\subsection{Computational efficiency}
\label{timing}

\begin{figure}
\centering
\begin{minipage}{.5\textwidth}
  \centering
  \includegraphics[scale=0.33]
  {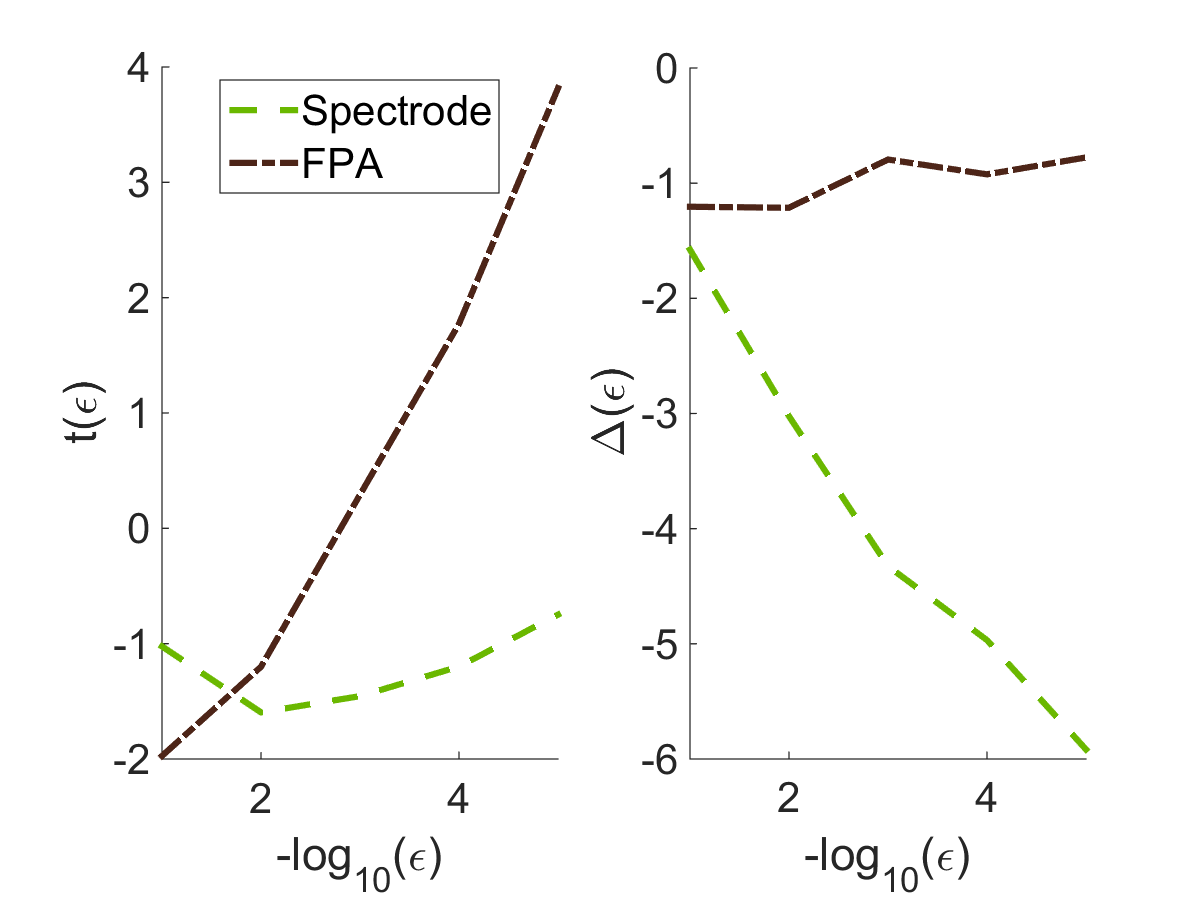}\\{\footnotesize (a) {\tt MP}}
\end{minipage}%
\begin{minipage}{.5\textwidth}
  \centering
  \includegraphics[scale=0.33]{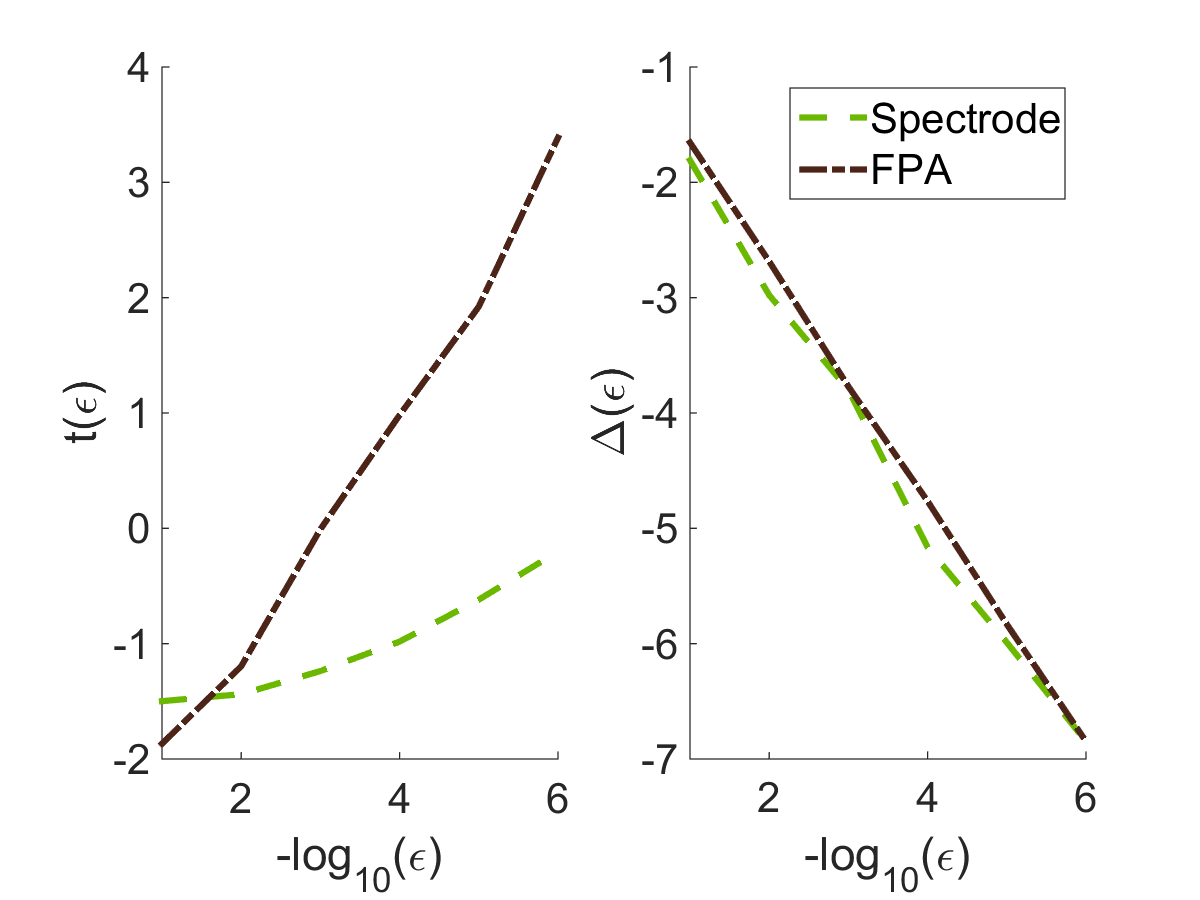}
  \\{\footnotesize (b) {\tt TwoPoint}}
\end{minipage}
\caption{Running time (base ten logarithm)  vs accuracy on two test problems (Section \ref{timing}). We show the log-running time ($t(\varepsilon,H,\gamma)$, left subplot in (a) and (b) ) and average accuracy ($\bar\Delta(\varepsilon)$, right subplot in (a) and (b)) of the methods as a function of the number of correct significant digits $k$ in the precision parameter $\varepsilon = 10^{-k}$. Methods: \textsc{Spectrode}  - dashed, fixed-point (FP), - dash - dotted. }
\label{fig:timing}
\end{figure}

We now establish that \textsc{Spectrode} is computationally efficient. We compare running times with \textsc{FPA} and find that for high precision problems on dense grids, \textsc{Spectrode} significantly outperforms \textsc{FPA}. 

\subsubsection{Test problems and parameters} We use the same test problems, {\tt MP} and {\tt TwoPoint}, and the same parameters ($J,\gamma,q$), as in the previous section. For a specified set of inputs $H$, $\gamma$, and accuracy $\varepsilon$, \textsc{Spectrode} produces density estimates $\hat f(x_{i},\varepsilon)$ on a grid $x_i$ $i=1, \ldots, I$. We record the running time $t(\varepsilon,H,\gamma)$ of the algorithm, defined as the base ten logarithm of seconds to completion. Times were measured on an Intel i7 2.4 GHz PC. The relative running times are relevant more generally for other systems. We also record the average accuracy in the density, defined as:
$ \bar\Delta(\varepsilon) = -\log_{10}\left(\sum_{i=1}^{I}|\hat f(x_{i},\varepsilon) - f(x_{i})|/I\right)$.
Here $f(x_i)$ is the true density which is available in both cases.  

We repeat this experiment for \textsc{FPA}, which is described later in Algorithm \ref{fp.alg} from Section \ref{accurate_density}. We record the running time $t_{\mathrm{fp}}(\varepsilon,H,\gamma)$ and accuracy $\bar\Delta_{\mathrm{fp}}(\varepsilon)$. To ensure comparability, we use the same grid $x_i$ that was produced by \textsc{Spectrode}. We set the accuracy parameter $\eta$ to $\eta = \varepsilon$. We emphasize that $\eta$ limits the precision due to the smoothing property of Stieltjes transforms (see Lemma \ref{stieltjes_bound}). Therefore, it should be of the same order as $\varepsilon$ to get the desired precision; this motivates our choice $\eta = \varepsilon$. Further, we apply an early stopping rule to the fixed-point algorithm, due to its long running time. For each grid point, we stop after $1/\varepsilon$ iterations. For this reason the fixed-point algorithm does not always achieve the required accuracy.  

\subsubsection{Results}
Figure \ref{fig:timing} shows the results, for {\tt MP} in the left panel and for {\tt TwoPoint} in the right panel. The running time $t(\varepsilon,H,\gamma)$ and the accuracy $\bar\Delta(\varepsilon)$ of the two methods are displayed as a function of the number of significant digits requested $-\log_{10}(\varepsilon)$.

For the test problem {\tt MP} in subplot (a), the number of significant digits requested varies from one to five, i.e. $\varepsilon = 10^{-1}, \ldots, 10^{-5}$.  The running time of \textsc{Spectrode} is below $0.5$ seconds, regardless of the accuracy requested, and produces the required average accuracy. The running time of \textsc{FPA} increases approximately linearly in $1/\varepsilon$, reaching $\sim 5000$ seconds for $\varepsilon =  10^{-5}$. At the same time, the average accuracy is always about one digit. In this example \textsc{Spectrode} is faster and more accurate at the same time. Had we stopped later, \textsc{FPA} would have taken even longer to converge.

This result is worth emphasizing: \textsc{Spectrode} is 1000 times faster and 1000 times more accurate than the fixed-point algorithm, at least for the highest precision $\varepsilon =  10^{-5}$.

For the test problem {\tt TwoPoint} in subplot (b), the number of significant digits requested now varies from one to six.  The running time of \textsc{Spectrode} is below one second, and it produces the required accuracy. The running time of \textsc{FPA} increases as $\varepsilon \downarrow 0$, reaching about $\sim$ 2000 seconds for the largest accuracy. The fixed-point algorithm also gives approximately the required accuracy. In this case, \textsc{Spectrode} is faster than \textsc{FPA} (by three orders of magnitude for $\varepsilon = 10^{-6}$) while producing the same accuracy.  These two examples show that \textsc{Spectrode} is fast and accurate, and compares favorably to the fixed-point algorithm.

\section{Theoretical Results}
\label{themethod}

\subsection{Background}

In this section we explain our method, and prove its convergence. We start with some background about limiting spectral distributions of large covariance matrices. Chapter 7 of Couillet and Debbah's monograph \cite{couillet2011random} provides a good summary of the material presented here. Recall the model presented in the introduction: $\mathbf{X}$ is $n \times p$, of the form $\mathbf{X} = \mathbf{Y}{\bf\Sigma}_p^{1/2}$, where the entries of $\mathbf{Y}$ are iid with mean zero and variance one. We take a sequence of such problems with $p,n$ growing to infinity such that $p/n \to \gamma>0$. The population SD of the deterministic ${\bf\Sigma}_p$ converges to the limit PSD $H$. 

The Marchenko-Pastur theorem (see \cite{marchenko1967distribution, silverstein1995strong}) states that the empirical SD of the sample covariance matrix $\widehat{{\bf\Sigma}} = {n}^{-1} \mathbf{X}^\top \mathbf{X}$ converges almost surely to a distribution $F$. Denote the imaginary part of $z \in \mathbb{C}$ by $\mathrm{Imag}(z)$  and the upper half of the complex plane by $\mathbb{C}^+ = \{z \in \mathbb{C}: \mathrm{Imag}(z)>0 \}$. If $m(z)$ denotes the Stieltjes transform of $F$, defined for $z \in \mathbb{C}^+$ as:
\begin{equation}
m(z) = \int \frac{\, dF(x)}{x - z}, 
\label{ST}
\end{equation}

 and $v(z)$ is the companion Stieltjes transform defined on $\mathbb{C}^+$ by the equation
 
\begin{equation}
\label{dual.ST}
\gamma\left(m(z)+1/z\right) = v(z)+1/z,
\end{equation}
 
   then it is shown in \cite{silverstein1995analysis} that $v(z)$ is the unique solution with positive imaginary part of the Silverstein equation :
\begin{equation}
\label{silv.eq}
-\frac{1}{v(z)} = z - \gamma \int \frac{t \, dH(t)}{1 + tv(z)},  \, z \in \mathbb{C}^+.
\end{equation}

This equation links the limit PSD $H$ to the limit ESD $F$. The function $v$ is analytic in the upper half $\mathbb{C}^+$ of the complex plane. Marchenko and Pastur  \cite{marchenko1967distribution} obtained a more general, but more complicated, form of this equation. The present form is due to J.W. Silverstein and appeared in  \cite{silverstein1992signal}. 

Our problem is to compute $F$ from $H$. For this it is enough to find $v(z)$, and thus $m(z)$,  for $z$ on a grid close to the real axis. Then, since $F$ has a density $f$, as shown in \cite{silverstein1995analysis}, by the inversion formula for Stieltjes transforms, we have the limit

\begin{equation}
f(x) = \frac{1}{\pi} \lim_{\varepsilon \to 0}\mathrm{Imag}\{ { m (x + i \varepsilon)}\}.
\label{stieltjes_to_density}
\end{equation}

This limit is valid at all points $x$ where the density $f(x)$ exists\footnote{The result of Silverstein and Choi \cite{silverstein1995analysis} is stronger. It also states that the density $f(x)$ is the imaginary part of $m(x)$, defined as the limit of the Stieltjes transform as $z \to x$; and the associated $v(x)$ is in fact the solution of the Silverstein equation \eqref{silv.eq} with $z=x$. While this is in fact an exact equation for the density of the ESD, we will not use it in the current paper. We will instead rely on the Silverstein equation on a grid close to the real axis. The reason is that we use FPA as a starting point of our method, and FPA is only known to converge for $z \in \mathbb{C}^+$, in the interior of the upper halfplane.}. For $\gamma<1$, the density exists for all $x$, while for $\gamma>1$ it exists for all $x$ except for $x=0$. In the latter case $F$ has an easily computable point mass at $x=0$. Numerically it is natural to use the approximation $\hat f = \mathrm{Imag}\{  m (x + i \varepsilon)\}/\pi$ for some small $\varepsilon>0$. There may be more efficient methods for interpolating $m (x + i \varepsilon)$, but those are beyond our scope.

In general there are many solutions to \eqref{silv.eq} with non-positive imaginary part. Indeed, for $H$ a finite mixture of point masses, $H = \sum_{i=1}^{J} w_i \delta_{t_i}$, the Silverstein equation becomes

\begin{equation}
\label{MP.eq.disc}
-\frac{1}{v(z)} = z - \gamma\sum_{i=1}^{p} \frac{w_it_i}{1 + t_iv(z)},  \, z \in \mathbb{C}^+.
\end{equation}

This is generally equivalent to a polynomial equation of degree $p+1$, and hence it has $p+1$ complex roots, compare \cite{rao2008polynomial}. The desired solution will ``track'' one of the roots as a function of $z$. However, finding the right solution by root tracking is not feasible in general for large $p$. There does not appear to be a way to efficiently compute the coefficents of the polynomial. Indeed, those coefficients involve all symmetric sums of the eigenvalues $t_i$, and computing these terms seems prohibitively expensive. We will take a different approach. 

\subsection{Our approach}

We differentiate the fixed-point equation  \eqref{MP.eq.disc} in $z$, and solve for $v'$. These steps yield the following ordinary differential equation for $v$:

\begin{equation}
\frac{dv}{dz} = \mathcal{F}(v) := \frac{1}{\frac{1}{v^2} -\gamma\sum_{i=1}^{J} \frac{w_it_i^2}{(1 + t_iv)^2}}, \mbox{ } \hat v(z_0) = \hat v_0.
\label{ode}
\end{equation}

A high-accuracy starting point for the ODE can be found by running the fixed-point algorithm once, at a point $z_0=x_0 + i\varepsilon$ near the real axis. Then, the ESD can be computed at other real values $x$ by solving the ODE on the line $x + i\varepsilon$, for the fixed $\varepsilon$ and varying $x \in \mathbb{R}$. Solving the ODE turns out to be much more convenient than solving the original equation repeatedly for each new point $x + i\varepsilon$.  The reason is that the limit spectral density is smooth, and the Stieltjes transform provides further smoothing. Our ODE uses this smoothness for efficient computation. This is in contrast to \textsc{FPA}, which re-runs the entire fixed-point iteration at each nearby point and does not exploit the smoothness. Using the smoothness via the ODE is the key inspiration behind our approach\footnote{This ``spectral ODE'' was also the source of the name \textsc{Spectrode}.}. 


Since the ODE was obtained by differentiating \eqref{MP.eq.disc}, it has at least one solution. We will show in our proof that there is only one solution in the range of interest. Then, once we obtain a numerical solution $\hat v(x)$ to the ODE, we could define $\hat f$ directly based on the explicit formulas \eqref{dual.ST} and \eqref{stieltjes_to_density}. This direct method already leads to a relatively good solution which provably converges to the right answer as $\varepsilon \to 0$. 

However, for small $\varepsilon$ this direct method has numerical problems caused by irregularities in the density at the edges of the support. Indeed, as shown in Figure \ref{two_examples}, the density exhibits a square root behavior at the boundary of the support. If implemented naively, these square root irregularities can cause difficulties to the ODE solver.  We will avoid these difficulties by a more sophisticated method, which first finds the edges of the spectrum by exploiting the theory of Silverstein and Choi \cite{silverstein1995analysis}, and later solves the ODE only within the support of $F$.

In brief, Silverstein and Choifind the support of the spectrum in the following way. They consider the Silverstein equation \eqref{silv.eq}, which defines the companion Stieltjes transform $v$ implicitly as a function of $z$. They observe that the same equation defines a function $z(v)$:
\begin{equation}
\label{ST_inverse}
z(v) = -\frac{1}{v}  + \gamma\sum_{i=1}^{J} \frac{w_it_i}{1 + t_iv}.
\end{equation}

 They prove that the support can be obtained by analyzing the monotonicity of $z$. Specifically, they show that the real intervals $v \in (a,b)$ where $z(v)$ is increasing, i.e. $z'(v)>0$, are precisely those, whose image under $z(v)$ (i.e. $(z(a),z(b))$) is the complement of the support of the distribution $\underline{F}$. Here $\underline{F}$ is the limit ESD of $n^{-1} \mathbf{X} \mathbf{X}^\top$. Since $\underline{F}$ is directly related to $F$, see \eqref{companion}, in theory this is enough to find the support. We will exhibit a computationally accurate method to approximate the values of the increasing intervals $(a,b)$, and prove its correctness. 

Algorithm \ref{ode_alg} below contains pseudocode for \textsc{Spectrode}. The full details are provided in the next sections.

\begin{algorithm}
\caption{\textsc{Spectrode}: computation of the limit ESD }\label{ode_alg}
\begin{algorithmic}[1]
\Procedure{Spectrode}{}
\BState \textbf{input}
\State $t_1, \ldots, t_J \gets$ positive eigenvalues
\State $w_1, \ldots, w_J \gets$ weights $w_i>0, \sum w_i = 1$
\State $\gamma \gets $ aspect ratio $\gamma \neq 1$
\State $\varepsilon \gets$ precision parameter
\BState \textbf{begin}
\State With accuracy $\varepsilon>0$, find all intervals $(a_k,b_k)$ where $z(v)$ \eqref{ST_inverse} is increasing ($a_k<b_k<a_{k+1}$)
\State Define the support intervals $[\hat l_k(\varepsilon),\hat u_k(\varepsilon)]$:
\If{$\gamma <1$}  
\State set $\hat l_k(\varepsilon) = z(b_{k})$ and $\hat u_k(\varepsilon) = z(a_{k+1})$ for all $k \le J-1$
\Else 
\State set $\hat l_1(\varepsilon) = z(b_J)$, $\hat l_k(\varepsilon) = z(b_{k-1})$ for all $2 \le k \le J-1$, and $\hat u_k(\varepsilon) = z(a_{k})$ for all $k \le J-1$
\EndIf
\State Set $\hat K(\varepsilon)$ to the number of intervals $[\hat l_k(\varepsilon), \hat u_k(\varepsilon)]$
\For{intervals $[\hat l_k(\varepsilon), \hat u_k(\varepsilon)]$}
\State Approximate by $\hat v_k$ the value $v_k = v(\hat l_k(\varepsilon)+ i \delta)$; $\delta = \varepsilon^2$ using FPA (Alg \ref{fp.alg}) with accuracy $\eta = \varepsilon$.
\State Define a uniform grid $\hat l_k = x_{k0} < \ldots < x_{kM} = \hat u_k$ with $\lceil\varepsilon^{-1/2}\rceil$ elements
\State Solve the ODE \eqref{ode} starting at $\hat v_k$ to find the values $\hat v(x_{kj}+ i \delta)$
\State Compute $\hat f(x_{kj},\varepsilon) = \mathrm{Imag}\{\hat m(x_{kj}+ i \delta)\}/\pi$, with  $\hat m$ from \eqref{dual.ST}
\EndFor
\BState \textbf{return}  $\hat K(\varepsilon)$; support intervals $[\hat l_k(\varepsilon), \hat u_k(\varepsilon)]$. Within estimated support intervals, define $\hat f(x)$ by linear interpolation. Outside the estimated support define $\hat f(x)=0$.
\EndProcedure
\end{algorithmic}
\label{MP_solve}
\end{algorithm}

\subsection{Correctness of \sc{Spectrode}}

\textsc{Spectrode} has an user-adjustable accuracy parameter $\varepsilon>0$. Here we show that as $\varepsilon\to0$, the output of the algorithm converges to the correct limiting values. 

\textsc{Spectrode} has the following two main steps:

\begin{enumerate}
\item Find the support of the distribution, as a union of compact intervals.
\item Compute an approximation of the density on the intervals inside the spectrum.
\end{enumerate}

We will analyze these two parts separately, and  give our main results in Theorems \ref{endpoints_accurate} and \ref{density_approx}. We will focus on the case $\gamma <1$; the case $\gamma>1$ is similar and therefore omitted.  

\subsubsection{Summary of Silverstein and Choi's results}

We rely in an essential way on the results of Silverstein and Choi in \cite{silverstein1995analysis}. For the reader's convenience, we summarize below the results we need.  For any cumulative distribution function $G$ on the real line, define the complement of the support of $G$, $S_G^c$, by $S_G^c = \{ x \in \mathbb{R}: \text{ there is an open neighborhood $N$ of $x$, such that $G(x)$ is constant on $N$} \}.$

The support of a distribution function $G$ is defined as $S_G = \mathbb{R} \setminus S_G^c$. The companion distribution function $\underline{F}$ is the limit ESD of $n^{-1} \mathbf{X} \mathbf{X}^\top$, and satisfies

\begin{equation}
\underline{F} = \gamma F + (1-\gamma) I_{[0,\infty)}.
\label{companion}
\end{equation}

We recall some claims, some of them stated informally earlier in the paper. 
\begin{lemma}[Silverstein and Choi \cite{silverstein1995analysis}] Let $F$ be the limit ESD of covariance matrices with limit PSD $H$, and aspect ratio $\gamma<1$. It holds that:
\begin{enumerate}
\item $F$ has a continuous density $f(x)$ for all $x$.

\item $f$ is analytic in the neighborhood of any point $x$ such that $f(x)>0$.
\item Let $B = \{m \in \mathbb{R}: m\neq0, -m^{-1} \in S_H^c\}$. Then $m$ belongs to $B$ iff $z(m)$ belongs to $S_{\underline F}^c$ and $z'(m)>0$. This characterizes the support of $\underline F$ and thus also that of $F$.

\item Suppose the PSD is an atomic distribution with $J$ point masses. The number of disjoint intervals $(a,b)$ in $S_F^c$ such that $a,b \in S_F$ is at most $J-1$. Therefore the support is the union of at most $J$ disjoint compact intervals. 
\end{enumerate}
\label{silv_choi}
\end{lemma}

The following  is a restatement of Lemma 6.2 in \cite{bai2009spectral}, see also Theorem 7.5 in \cite{couillet2011random}.

\begin{lemma}[consequences of \cite{silverstein1995analysis}, see Lemma 6.2 in \cite{bai2009spectral}] Suppose the PSD $H$ is an atomic distribution with $t_i>0$ for all $i=1,\ldots,J$, and $\gamma < 1$. Then 

\begin{enumerate}
\item $F$ is compactly supported.
\item The density $f(x)$ equals zero in some right-neighborhood $(0,a)$ of 0.
\item Suppose $(c,d)$ is an interval where $z'(v)>0$, but $z'(c) = z'(d)=0$, and $z'(v) <0$ in some neighborhoods $(c-\delta,c),(d,d+\delta)$. Then the interval $(z(c),z(d))$ forms one connected component of $S_F^c$, and all maximal compact intervals in $S_F^c$ have this form.
\item $z'(v)>0$ for all $v \in (0,\infty)$. Further, $z(v) <0$ for all $v \in (0,\infty)$. 
\item  Define $D = -1/\min_j{t_j}<0$.  There is a finite constant $b<D$, such that $z'(b)=0$, while $z'(v)>0$ for $v \in (-\infty,b)$ and $z'(v)<0$ for $v \in (b,D)$. Then $l_1 = z(b)$ is the lowest endpoint in the support of $F$.
\end{enumerate}
\label{easy_lemma}
\end{lemma}

Therefore, the support of $F$ is the union of some intervals $[l_i,u_i]$, where  $0<l_1<u_1<\ldots<l_K<u_K$. The density $f(x)=0$ on $(0,l_1]$, $[u_i,l_{i+1}]$ for all $i$, and $[u_K,\infty)$. Within the intervals $[l_i,u_i]$, the density $f(x)$ is usually strictly positive. However, there are cases in which the density $f(x)=0$ at isolated points within $[l_i,u_i]$. This can happen if $[l_i,u_i]$ arises from a ``merge'' when two intervals corresponding to neighboring population eigenvalues ``just touch''. We emphasize that in this case we consider $[l_i,u_i]$ as one component of the support of $F$. 

\subsubsection{Correctness of the support}

Here we explain in detail the steps to find the support of $F$, and prove their correctness:

\begin{theorem}{\bf Correctness of $\hat{K}(\varepsilon)$,  $\hat{l}_k(\varepsilon)$, and $\hat{u}_k(\varepsilon)$: }
Consider the numerical approximations $\hat{K}(\varepsilon)$,  $\hat{l}_k(\varepsilon)$, $\hat{u}_k(\varepsilon)$ outlined in Algorithm \eqref{ode_alg} and described in detail below. Then, as $\varepsilon \to 0$:
\begin{enumerate}
\item The number of disjoint intervals is correctly identified: $\lim _{\varepsilon \to 0} \hat{K}(\varepsilon) = K. $
\item The endpoints of the support are accurately approximated: $\lim _{\varepsilon \to 0} \hat l_k(\varepsilon) = l_k$, $\lim _{\varepsilon \to 0} \hat u_k(\varepsilon) = u_k$. 
\end{enumerate}
\label{endpoints_accurate}
\end{theorem}

This theorem is a consequence of Lemmas \ref{accurate_1_endpoint_lemma} and \ref{accurate_endpoints_lemma} below. 

Our strategy, following Lemmas  \ref{silv_choi} and \ref{easy_lemma}, is to find the intervals where $z$ is increasing; or specifically the points where it switches monotonicity. Once we find such a switch point $\hat a$, we will define $z(\hat a)$ as an approximate endpoint. In detail, by Lemma \ref{easy_lemma}, claim 5, $l_1 = z(b)$ is the lowest endpoint in the support of $F$, where $b$ is the largest point such that $z$ is increasing on $(-\infty,b)$. Therefore, in our algorithm \eqref{ode_alg}, line 4, we set the leftmost interval where $z$ is increasing as $(a_1,b_1)$, with $a_1 = -\infty$, and $b_1 = \hat b$, where $\hat b(\varepsilon)$ is the numerical estimate of $b$ found below. 

The numerical estimate of $b$ is found in the following way. Recall $D = -1/\min_j{t_j}<0$ and consider a grid $u_0 <u_1<\ldots<u_N = B$ depending on $\varepsilon$. Let $h(\varepsilon)$ be a function such that $\max_i |u_{i+1}-u_i| \le h(\varepsilon)$, and $1/|u_0| <h(\varepsilon)$. Find the smallest index $i$ such that $z'(u_i)<0$ and let $\hat b(\varepsilon) = u_i$. If there is no such index, then let $i=N$. Define $\hat  l_1(\varepsilon)  = z(\hat b(\varepsilon))$ as a numerical approximation of $l_1 = z(b)$. The following lemma ensures the convergence of this procedure.

\begin{lemma} {\bf Correctness of $\hat{l}_1(\varepsilon)$: }
Suppose the function $h(\varepsilon)$ has limit $\lim_{\varepsilon \to 0}h(\varepsilon)=0$. Then  $\lim _{\varepsilon \to 0} \hat l_1(\varepsilon) = l_1$.
\label{accurate_1_endpoint_lemma}
\end{lemma}

\begin{proof}As  $h(\varepsilon) \to 0$, for sufficiently small $\varepsilon$ we will have $u_0 < b < u_N$. Therefore, for some $i$ it will be true that $u_{i-1} \le b < u_i$. By claim 5 of Lemma \ref{easy_lemma}, this will be the smallest index $u_i$ for which $z'(u_i)<0$, and hence $\hat b(\varepsilon) = u_i$. Now by assumption $u_i- u_{i-1} \le h(\varepsilon)$, hence $|\hat b(\varepsilon) - b| \le h(\varepsilon)$. 

This shows that as $\varepsilon\to0$, we have $\hat b(\varepsilon) \to b$. By inspection, $z$ is continuous on $(-\infty,D)$, hence $\hat  l_1(\varepsilon) = z(\hat b(\varepsilon)) \to z(b) = l_1$, as claimed.
\end{proof}
In practice, we choose the grid $u_1,\ldots,u_N$ using an iterative doubling. We first set it to be an equi-spaced grid on $[D-1,D-\varepsilon]$ with $N(\varepsilon) = \lfloor 1/\varepsilon \rfloor$ elements. If this grid doesn't contain an element with $z'(u_i)>0$, we switch to a grid on $[D-2, D-\varepsilon/2]$ with $2N(\varepsilon)$ elements. We iterate this process until we find an index with $z'(u_i)>0$ and proceed as above. This procedure implicitly defines $h(\varepsilon) \approx \varepsilon$, but more general choices also work as long as $\lim_{\varepsilon \to 0}h(\varepsilon)=0$. 

We have shown that the numerical approximation to the lowest endpoint of $F$ converges. Similarly, we define the remaining endpoints. By Lemma \ref{easy_lemma} claim 4, there is no need to look at the interval $(0,\infty)$. Let us define a new grid $D = y_0 < y_1 <\ldots<y_M=0$ depending on $\varepsilon$, such that $\max_i |y_{i+1}-y_i| \le h(\varepsilon)$. Here $h$ is again a function such that $\lim_{\varepsilon \to 0}h(\varepsilon)=0$, and in our implementation we choose $h(\varepsilon) \propto \varepsilon^{1/2}$. Without loss of generality we can assume that the $y_i$, $i \in 1,\ldots, M-1$, are disjoint from the finitely many values $-1/t_j$, so that $z, z'$ are well-defined on the grid; but see the practical note at the end of this section. 

Consider the sign of the sequence $z'(y_j)$. Assume without loss of generality that the smallest $t_i$ is $t_1$. Clearly for $v$ near $D=-1/t_1$, the dominating term in $z'$ is $-w_1\{t_1v/(1 + t_1v)\}^2$, and this tends to $-\infty$ as $v \to D, v>D$. Hence, for $y_1$ sufficiently close to $D$, we can assume $z'(y_1)<0$. The sequence $z'(y_j)$ then starts out negative. Denote by $i_1$ the first index where it switches sign. Define the sequence of grid indices $i_1 < i_2  < \ldots$ inductively, with $i_{k+1}$ the first index $j>i_k$ where the sign of $z'(y_{j})$ differs from the sign of $z'(y_{i_{k}})$. 

Then, let us define for all $k$, $a_k = y_{i_{2k-1}}$, $b_k = y_{i_{2k}-1}$. Thus $a_k,\ldots,b_k$ are the ranges of indices where $z'>0$. As described in line 5 of Algorithm \eqref{ode_alg}, set $\hat l_k(\varepsilon) = z(b_{k})$ for $2 \le k \le J-1$ and $\hat u_k(\varepsilon) = z(a_{k+1})$ for all $k \le J-1$. Finally, set $\hat K(\varepsilon)$ as the number of intervals $(a_k,b_k)$ constructed. The lemma below ensures the convergence of this procedure.

\begin{lemma} {\bf Correctness of $\hat{K}(\varepsilon)$,  $\hat{l}_k(\varepsilon)$, ($k \ge 2$), and $\hat{u}_k(\varepsilon)$: }
Suppose the function $h(\varepsilon)$ has limit $\lim_{\varepsilon \to 0}h(\varepsilon)=0$. Then the number of intervals is correctly identified in the high-precision limit: $\lim _{\varepsilon \to 0} \hat{K}(\varepsilon) = K$. Further, in addition to $l_1$, the approximations to the other endpoints of the support converge to the right answers: $\lim _{\varepsilon \to 0} \hat l_k(\varepsilon) = l_k$ for all $k \ge 2$, and $\lim _{\varepsilon \to 0} \hat u_k(\varepsilon) = u_k$ for all $k$.
\label{accurate_endpoints_lemma}
\end{lemma}

\begin{proof}
This is analogous to the previous proposition. Suppose $(c,d)$ is an interval where $z'(v)>0$, but $z'(c) = z'(d)=0$, and $z'(v) <0$ in some neighborhoods $(c-\delta,c),(d,d+\delta)$. By claim 3 of Lemma \ref{easy_lemma}, the interval $(z(c),z(d))$ forms one connected component of the complement of the support of $F$, and all maximal compact intervals in $S_F^c$ have this form. In addition there is an unbounded interval $(u_K,\infty)$ in the complement of $F$, which is the image of the largest interval $(c,0)$ on which $z'>0$. Therefore we must find all increasing intervals, with special attention to the last one. Since we already found $\hat l_1$ in the previous lemma, we exclude the interval $(-\infty,l_1]$.

Consider first an interval $(c,d)$ with the properties above, where $d<0$. Since the grid spacings $|u_{i+1}-u_i|\le h(\varepsilon) \to 0$, for small $\varepsilon$, we will have $c-\delta <y_i < c < y_{i+1}$ for some $i$. Therefore, $z'(y_i)<0$ and $z'(y_{i+1})>0$.  This shows that the sign of the sequence $z'(y_j)$ switches at $i+1$. 

Similarly, the sign of the sequence switches at the index $j+1$ for which $ y_j < d < y_{j+1} < d+ \delta$. For sufficiently small $\varepsilon$ each switch in the signs will correspond to exactly one such interval $(c,d)$, because there are finitely many true switches. Hence our algorithm will choose $a_k = y_{i+1}$, $b_k = y_{j}$.

Then it's clear that $0 \le a_k - c \le y_{i+1} - y_i \le h(\varepsilon)$. Therefore $\lim_{\varepsilon \to 0} a_k(\varepsilon) = c$. Now by claim 3 of Lemma \ref{easy_lemma}, $c$ equals the pre-image under $z$ of the upper endpoint $u_l$ of some support interval, i.e. $z(c) = u_l$. Since each switch in the signs of $z'(y_i)$ corresponds to exactly one interval $(c,d)$, this means that the sorted values $c_1<d_1<c_2<d_2<\ldots$ correspond to the pre-image under $z$ of $u_1<l_2<\ldots$, i.e. $z(c_i)=u_i$, $z(d_{i})=l_{i+1}$ for all $i$.

This shows that $\lim_{\varepsilon \to 0} a_{k+1}(\varepsilon) = z^{-1}(u_k)$ for all $k \le K-1$. It's easy to see that $z$ is continuous at $a_{k+1}$, because the only points of discontinuity are at the values $-1/t_i$. By continuity $\lim_{\varepsilon \to 0} z(a_{k+1}(\varepsilon)) = u_k$, i.e. $\lim_{\varepsilon \to 0} \hat u_k(\varepsilon)= u_k$, for $k \le K-1$, as required. Similarly $\lim _{\varepsilon \to 0} \hat l_k(\varepsilon) = l_k$ for all $k \ge 2$.

To show that the highest support endpoint converges, i.e., $\lim_{\varepsilon \to 0} \hat u_K(\varepsilon)= u_K$, we must study the largest interval $(c,0)$ where $z'>0$. The reason is that, as it's easy to see, $z'>0$ in a small neighborhood $(-\eta,0)$ of zero. Therefore, by Lemma \ref{easy_lemma}, the upper endpoint of the support of $F$ will be the smallest $c$ such that $z'>0$ on $(c,0)$. The proof of the convergence in this case is very similar to the analysis presented above, and therefore omitted.
\end{proof}

Lemmas \ref{accurate_1_endpoint_lemma} and \ref{accurate_endpoints_lemma} together imply Theorem  \ref{endpoints_accurate}. 
Our practical implementation of the algorithm is a bit more involved. We consider all intervals $J_i = (-1/t_i,-1/t_{i+1})$ one-by-one. It is beneficial to break down the search for increasing intervals to separate intervals $J_i$ because  - as it is easy to see - $z(v)$ has at most one increasing interval within each $J_i$. Further, $z(v)$ has no singularities within any $J_i$. These properties ensure added stability for finding the support.

\subsubsection{Correctness of the density}
\label{accurate_density}

In this section we prove the accurate computation of the density: 

\begin{theorem}{\bf Correctness of $\hat{f}(x,\varepsilon)$:}
Consider the numerical approximation to the density $\hat{f}(x,\varepsilon)$, outlined in Algorithm \eqref{ode_alg} and described in detail below. Then, as $\varepsilon \to 0$, the approximation converges to the true density uniformly over all $x$: $\sup_{x\in \mathbb{R}}|\hat{f}(x,\varepsilon) - f(x)| \to 0$. 
\label{density_approx}
\end{theorem}

This theorem is proved in at the end of this section, using the tools developed in in the following lemmas. The approximation to the density is defined in several stages, which are outlined below. For the reader's convenience, we provide Table \ref{approx_stages} summarizing the notation and definitions for each stage.

\begin{table}[h] 
\footnotesize
\centering
\caption{Definitions used in the proof of Theorem \ref{density_approx}}

\begin{tabular}{|l|l|l|l|l} 
\cline{1-4}
Name                                     & Definition                   & Defined in               & Analyzed in Lemmas                            &  \\ \cline{1-4}
$x_j(\varepsilon)$                       & grid                         &  \eqref{grid_def}                            &                                                            &  \\
$f(x)$                                   & true density                 &        \eqref{stieltjes_to_density}                      & \ref{grid_converge_lemma} &  \\
$f(x_j(\varepsilon), \varepsilon)$       & solution to exact ODE        & \eqref{exact_ode_def}    & \ref{ode_unique} &  \\
$\tilde f(x_j(\varepsilon),\varepsilon)$ & solution to inexact start ODE      & \eqref{inexact_ode_def}  & \ref{inexact_ode}, \ref{ODE_approx}                  &  \\
$\hat f(x_j(\varepsilon),\varepsilon)$   & Euler's method approximation & \eqref{euler_method_def} & \ref{grid_converge_lemma}, \ref{ODE_approx}         & \\ \cline{1-4}
\end{tabular}
\label{approx_stages}
\end{table}

In the first stage, outside the support intervals $[\hat{l}_{k}(\varepsilon), \hat{u}_{k}(\varepsilon)]$ we define $\hat{f}(x, \varepsilon) =0$. We also set $\hat f (\hat{l}_{k}(\varepsilon),\varepsilon)=\hat f(\hat{u}_{k}(\varepsilon), \varepsilon) =0$. Since the approximated support converges, we see that the estimated density converges to zero uniformly outside of the support. All that is left is to handle the support intervals. 

Consider a support interval $[l,u]$, where $l<u$ are the true endpoints of a connected component of the support of $F$. Denote the estimated support by $[\hat l, \hat u]$, and let $\hat l = x_0 < x_1 < \ldots < x_M = \hat u$ be a uniformly spaced grid $x_i = x_i(\varepsilon)$ of length $M = M(\varepsilon)$ depending on $\varepsilon$, on which we will approximate the density $\hat{f}(x_j)$. In Algorithm \eqref{ode_alg}, we specified $M = \lceil\varepsilon^{-1/2}\rceil$ for concreteness. We will see in the proof that more general choices of grids work. For this reason, we will not specify the choice of the grid at the moment, and instead only require that the spacings tend to zero: $|x_{i+1}-x_{i}| \le h(\varepsilon)$, and $\lim_{\varepsilon \to 0} h(\varepsilon) = 0$.

Next, we reduce the approximation problem to the grid $x_i$. As explained in Algorithm \eqref{ode_alg}, for $x$ within the estimated support and not necessarily on the grid, define the linear interpolation $\hat{f}(x,\varepsilon) = \alpha \hat{f}(x_i,\varepsilon) + (1-\alpha) \hat{f}(x_{i+1},\varepsilon)$, where $x_i \le x < x_{i+1}$, and $x = \alpha x_i + (1-\alpha)x_{i+1}$. This ensures that the estimated density $\hat f(\cdot,\varepsilon)$ is continuous. With these definitions, we reduce uniform convergence over all $x$ to uniform convergence only on the grid.

\begin{lemma} To show the convergence in Theorem \ref{density_approx}, it is enough to show that the density approximations converge uniformly on the grid $x_i = x_i(\varepsilon)$, that is: 
\begin{equation}
\lim_{\varepsilon \to 0} \max_{0 \le i \le M(\varepsilon)}|\hat{f}(x_{i}(\varepsilon),\varepsilon)- f(x_{i}(\varepsilon))| = 0.
\label{grid_converge}
\end{equation}
\label{grid_converge_lemma}
\end{lemma}

\begin{proof}
It is easy to check that by construction, $l_i \le  \hat l_i \le \hat u_i \le u_i$ for all support intervals. Therefore, we have for any $x$
$$
\hat{f}(x,\varepsilon) - f(x) =
\left\{
	\begin{array}{ll}
		0 & \mbox{\, if \, }  x \notin [l_i,u_i], \mbox{ for any } i \\
		-f(x) &\mbox{\, if \, }  l_i \le x \le \hat l_i, \mbox{ or } \hat u_i \le x \le u_i \mbox{ for some } i \\
		\hat{f}(x,\varepsilon) - f(x) &\mbox{\, if \, }  \hat l_i \le x \le \hat u_i,\mbox{ for some } i \\
	\end{array}
\right.
$$

The convergence claim made by Lemma \ref{grid_converge_lemma} is clear in the first case. In the second case, note that there are only finitely many support intervals. Therefore it is enough to show $\lim_{\varepsilon \to 0} \sup_{x:l_i \le x \le \hat l_i}|f(x)|=0$ for all $i$, and the analogous statement for upper endpoints. We showed earlier in Proposition \ref{endpoints_accurate} that $\hat l_i \to l_i$. By continuity of $f$, this shows the desired claim $\sup_{x:l_i \le x \le \hat l_i}|f(x)| \to 0$ for the second case. 

The third case is the most interesting one. Consider any $x$ such that $\hat l_i \le x \le \hat u_i$. There are two neighbors in the grid such that $x_i(\varepsilon) \le x < x_{i+1}(\varepsilon)$. By the triangle inequality, we can bound

$$|\hat{f}(x,\varepsilon) - f(x)| \le |\hat{f}(x,\varepsilon) - \hat{f}(x_i,\varepsilon)| + |\hat{f}(x_i,\varepsilon) - f(x_i)| + |f(x_i) - f(x)| .$$

Recall that the maximum spacing was bounded: $|x_{i+1}-x_i| \le h(\varepsilon)$. Let us denote a modulus of continuity for a function $g$ by $\omega$. This function $\omega$ enjoys $|g(x) - g(y)| \le \omega(|x-y|,g)$ for any $x,y$. Taking the maximum over all $x \in S_i$ (where $S_i = [\hat l_i, \hat u_i]$) in the previous display, we obtain:

$$\sup_{x\in S_i}|\hat{f}(x,\varepsilon) - f(x)| \le \omega(h(\varepsilon), \hat f) + \max_{0 \le i \le M(\varepsilon)}|\hat{f}(x_i,\varepsilon) - f(x_i)| + \omega(h(\varepsilon),  f). $$

Since $f,\hat f$ are continuous and compactly suppported, they are uniformly continuous. Therefore, as $h(\varepsilon) \to 0$, we get  $\omega(h(\varepsilon),  f) \to 0$, and similarly for $\hat f$. Assuming \eqref{grid_converge}, this yields the desired claim:
$$\lim_{\varepsilon \to 0} \sup_{x\in S_i}|\hat{f}(x,\varepsilon)- f(x)| = 0.$$
\end{proof}

We will now focus on showing the convergence on the grid. The numerical approximation is found using an ordinary differential equation, whose starting point is obtained via the fixed-point algorithm (FPA). In \cite{couillet2011deterministic} FPA is presented for a more general class of problems; for the reader's convenience, we describe the special case needed in Algorithm \ref{fp.alg}. 

\begin{algorithm}
\caption{\textsc{FPA}: Fixed-Point Algorithm}\label{fp.alg}
\begin{algorithmic}[1]
\Procedure{\textsc{FPA}}{}
\BState \textbf{input}
\State $\textit{H} \gets \text{population eigenvalue distribution}$
\State $\textit{$\gamma$} \gets \text{ aspect ratio}$
\State $\eta \gets$ accuracy parameter ($>0$)
\State $z \gets \text{complex argument $\in \mathbb{C}^+$. If the argument $x$ is real, then $z \gets x + i \eta^2$}$
\BState \textbf{initialize}:
\State $v_0 \gets -1/z$
\State $n \gets 0$
\State $h(v) :=  z - \gamma \int \frac{t}{1 + tv}\, dH(t)$

\BState \textbf{while} {$ |1/v_n + h(v_n)| > \eta$}
\State ${v_{n+1}} \gets -1/h(v_n)$
\State $n \gets n+1$.
\BState \textbf{end};
\State $m_n \gets \gamma^{-1}v_n + (\gamma^{-1}-1)/z$ 
\BState \textbf{return} $\hat v(z,\eta) = v_n$; $\hat f(x,\eta) = \mathrm{Imag}(m_n) /\pi$, where $x =  \mathrm{Re}(z)$
\EndProcedure
\end{algorithmic}
\end{algorithm}

The fixed-point algorithm is a method for solving the Silverstein equation. It is based on the observation that \eqref{MP.eq.disc} is a fixed-point equation $v = -1/h(v;z)$, for any given $z$. Then one defines a starting point $v_0=-1/z$, and iterates $v_{n+1} = -1/h(v_{n};z)$ until the convergence criterion $ |1/v_n + h(v_n)| \le \eta$ is met. Let $\hat v(z,\eta)$ be the solution produced by the fixed-point algorithm.  It was shown in \cite{couillet2011deterministic} that, for any fixed $z \in \mathbb{C}^+$,  $\hat v(z,\eta)$ converges to the unique solution of the Silverstein equation \eqref{MP.eq.disc} with positive imaginary part, as $\eta \to 0$: $\hat v(z,\eta) \to v(z)$. 

The accuracy parameter $\eta$ in the fixed-point algorithm is important. If the algorithm is used to approximate the density of the ESD at $x$, with accuracy $\eta$, as in the section comparing the different methods, we run \textsc{FPA} at $z =x + i \eta^2$. The scaling $\eta^2$ is motivated by Lemma \ref{stieltjes_bound}. The proof of this lemma shows that the true solution at imaginary part $\varepsilon$ guarantees an accuracy $\varepsilon^{1/2}$. Therefore, to get accuracy $\eta$, we go down to imaginary part $\eta^2$. Further, we use the same threshold $\eta$ in the stopping criterion. In principle, these two parameters could be decoupled, but this simple choice suffices for our purposes. 

Therefore, to find the starting point, we use FPA for $z  =\hat l + i\delta(\varepsilon)$, where $\hat l$ is the approximation to the lower grid endpoint, and $\delta(\varepsilon)=\varepsilon^2$. The accuracy parameter $\eta$ is set to $\eta = \varepsilon$. This gives a starting point $\hat v_0 = \hat v(\hat l + i\delta(\varepsilon))$. For simplicity, we will first analyze, in Lemma \ref{ode_unique}, the case of an exact starting point $v_0 = v(\hat l + i\delta(\varepsilon))$. That is, we argue about the case when the solution of the Silverstein equation has been found \emph{exactly} by FPA. In Lemma \ref{inexact_ode} we extend the argument to the inexact starting point $\hat v_0$.

We call the \emph{exact ODE} Eq. \eqref{ode} on the interval $[\hat l,\hat u]$, started at the true solution $v_0 = v(\hat l + i\delta(\varepsilon))$.  The exact ODE has the right solution: 

\begin{lemma}{\bf Correctness of the exact ODE:} Consider the ODE \eqref{ode} for the complex-valued function $r$ of a real variable $x$
\begin{equation*}
\frac{dr}{dx} = \frac{1}{\frac{1}{r^2} -\gamma\sum_{i=1}^{J} \frac{w_it_i^2}{(1 + t_ir)^2}}, \mbox{ } r(x_0) = v_0.
\end{equation*}
Let the starting point be the exact solution $v_0 = v(x_0 + i\delta(\varepsilon))$. Then this equation has a unique solution on $[x_0,\hat u]$, and this solution is $r(x)  = v(x + i\delta(\varepsilon))$.
\label{ode_unique}
\end{lemma}

\begin{proof}
The ODE was obtained by differentiating the Silverstein equation \eqref{MP.eq.disc} for $v$. Since $r(x)  = v(x + i\delta(\varepsilon))$ obeys the Silverstein equation and also satisfies the starting condition, it is clearly a solution. 

To show that the solution is unique, we appeal to basic results in ordinary differential equations. Specifically, consider the general ODE $y' = g(y)$, $y(x_0) = y_0$. It is well known (e.g. Theorem 7.4 on pp. 39 of the monograph by Hairer and Norsett \cite{hairer2009solving}) that the solution is unique on the open set $(x,y) \in U \subset \mathbb{R} \times \mathbb{C}$ where $g, g'$ are continuous. If started at any point $(x_0,y_0) \in U$, the solution can be continued to the boundary of {\it U}. 

In our case $g(r)$ is continuous on the entire image $v(\mathbb{C}^+) = \{y: y = v(z), \mbox{ for some } z \in \mathbb{C}^+\}$. Indeed, let $y_0$ be an arbitrary element of $v(\mathbb{C}^+) $, so that $y_0 = v(z_0)$ for some $z_0 \in \mathbb{C}^+$. Then by the definition of the ODE, $v'(z_0) = g(y_0)$. Now $v$ is analytic on $\mathbb{C}^+$, so clearly $v'(z)$ is well-defined and continuous at $z_0$. By the expression for $v'(z)$,  $v'(z) = 1/k(v)$ for some complex function $k$, we see that $v'(z) \neq 0$. Hence, by the inverse function theorem, $v$ is invertible near $y_0$, so that locally $z = v^{-1}(y)$ in a neighborhood of $y_0$. Therefore, locally near $y_0$, $g(y) = v'( v^{-1}(y)) $. This shows that $g$ is continuous near $y_0$. Hence, $g(y)$ is continuous on the entire image $v(\mathbb{C}^+)$. By a similar argument $g'(y)$ is continuous on the entire image $v(\mathbb{C}^+)$.

This shows that for our problem $U$ contains $\mathbb{R} \times v(\mathbb{C}^+) $. Clearly we start at a point $(x_0,v_0)$ in $U$. By the result cited above, the solution to the ODE is unique on the entire set $x \in \mathbb{R}$, and in particular on  $[x_0,\hat u]$, finishing the proof. 
\end{proof}

Next, we will argue that even with an inexact starting point for the ODE, the solutions are still nearly exact. Suppose that FPA produces an estimate $\tilde v_0 = \hat v(z,\eta)$ of $v_0$. We call the ODE \eqref{ode} started at $\tilde v_0$ the  \emph{inexact start ODE}.  The difference $c(\varepsilon) = \tilde v_0- v_0$ can be made arbitrarily small by taking $\eta$ sufficiently close to zero in FPA. The following lemma ensures that the solution to the inexact start ODE stays close to the true solution.

\begin{lemma}{\bf Correctness of the inexact start ODE:} Consider the ODE \eqref{ode} as given in Lemma \ref{ode_unique}, but started at $\tilde v_0= v_0 + c(\varepsilon)$, where $\tilde v_0$ is the solution produced by FPA for starting point $z  =\hat l + i\delta(\varepsilon)$ and sufficiently small $\eta$. Then, for sufficiently small $\varepsilon$, this inexact start ODE has a unique solution $\tilde r$ on $[\hat l, \hat u]$, which obeys $|\tilde r(x) - r(x)| = O(c(\varepsilon))$ uniformly over all $x \in [\hat l, \hat u]$.
\label{inexact_ode}
\end{lemma}

\begin{proof}
First we will show the uniqueness of the solution. By Proposition 1 of \cite{couillet2011deterministic}, the fixed-point Algorithm  \ref{fp.alg} started at $z  = \hat l + i\delta(\varepsilon)$, and with accuracy $\eta$, produces a solution $\tilde v_0 = \hat v(z,\eta)$ such that  $\hat v(z,\eta) \to v(z)$ as $\eta \to 0$. Therefore, for sufficiently small $\eta$, we can ensure that $\tilde v_0 - v_0 = O(c(\varepsilon))$ for an arbitary small $c(\varepsilon)$.

By the open mapping theorem, $v(\mathbb{C}^+)$ is an open set containing the true solution curve $r(x) = v(x + i\delta(\varepsilon))$. Hence, for sufficiently small $c$, $\tilde v_0 \in v(\mathbb{C}^+)$. Hence, $\tilde v_0 = v(a + ib)$ for some $a$ and $b>0$. Then, by the same argument as in Lemma \ref{ode_unique}, the solution $\tilde r$ exists and is unique, and is given by $\tilde r(x) = v(  (x - x_0) + a + ib)$. This solution exists for all $x$ and belongs to $v(\mathbb C^+)$. 

Next, we will use a general inequality for inexact start ODEs for the quantitative bound. The `fundamental lemma' (Theorem 10.2 of \cite{hairer2009solving}, pp 58) states: Consider the ODE $y' = g(y)$, and let $y,\tilde y$ be two solutions with starting points $y(x_0), \tilde y(x_0)$. Suppose $|g'(y)| \le L$ on a connected set $K_0$ containing the two solution curves $y,\tilde y$ for $x \in [x_0,x_M]$. Then the solutions $y,\tilde y$ are close to each other for all $x \in [x_0,x_M]$, specifically: $|y(x) - \tilde y(x)| \le |y(x_0) - \tilde y(x_0)| \exp((x-x_0)L)$.

We have shown in the proof of Lemma \ref{ode_unique} that $f'(y)$ is continuous in the image $v(\mathbb{C}^+)$.  Let $K_0$ be a compact connected subset of $v(\mathbb C^+)$ containing the solution curves $r(x),\tilde r(x)$ for $x \in [\hat l, \hat u]$ (which exist by the above argument). Then $f'$ is bounded by some constant $L$ on $K_0$. By the fundamental lemma, we have  $|\tilde r(x) - r(x)| \le c(\varepsilon) \exp((\hat u - \hat l) L)  = O(c(\varepsilon))$, as required. This finishes the lemma. 
\end{proof}

Next we introduce notations for the solutions of the two ODEs. As we discussed earlier, the grid $x_i = x_i(\varepsilon)$ is uniformly spaced on $[\hat l, \hat u]$:
\begin{equation}
\hat l = x_0 < \ldots < x_M = \hat u,
\label{grid_def}
\end{equation}

The solutions to the exact ODE \eqref{ode} in Lemma \ref{ode_unique} on the grid are $v(x_j + i\delta (\varepsilon))$.  We then define $m(x_j + i\delta (\varepsilon))$ according to \eqref{dual.ST}, and 

\begin{equation}
f(x_j(\varepsilon), \varepsilon) = \mathrm{Imag}(m(x_j + i\delta (\varepsilon)))/\pi.
\label{exact_ode_def}
\end{equation}

Similarly, with the solutions $\tilde v(x_j + i\delta (\varepsilon))$ to the inexact start ODE analyzed in Lemma \ref{inexact_ode} on the same grid define $\tilde m(x_j + i\delta (\varepsilon))$ accordingly, using \eqref{dual.ST}, and 
 
\begin{equation}
 \tilde f(x_j(\varepsilon),\varepsilon) =  \mathrm{Imag}(\tilde m(x_j + i\delta (\varepsilon)))/\pi.
\label{inexact_ode_def}
\end{equation}

Lemma \ref{inexact_ode} shows that 

\begin{equation}
|\tilde f(x_j(\varepsilon),\varepsilon) - f(x_j(\varepsilon),\varepsilon)|  = O(c(\varepsilon)).
\label{inexact_ode_bound}
\end{equation}

In particular, this bound highlights that the approximation $\tilde f$ is uniformly accurate over the grid $x_j$, as required by Lemma \ref{grid_converge_lemma}.

We show next that Euler's method for discretizing the inexact start ODE on the grid $x_i$ produces numerical approximations that converge as $\varepsilon \to 0$. In practice we use a higher order ODE solver, but for simplicity here we consider Euler's method.

Let $\hat v_j$ be the sequence produced by Euler's method for the inexact start ODE \eqref{ode} on the discretization $\hat l = x_0 < \ldots < x_M = \hat u$.  Define $\hat m_j = \gamma^{-1}\hat v_j + (\gamma^{-1}-1)/z_j$, where $z_j = x_j + i\delta (\varepsilon)$. Also define the density approximations 
\begin{equation}
\hat f(x_j,\varepsilon) = \mathrm{Imag}(\hat m_j)/\pi.
\label{euler_method_def}
\end{equation}

Then we have the following:

\begin{lemma} {\bf Euler's method approximates the true solution to the inexact start ODE:} Consider a fixed $\varepsilon>0$, and suppose that $\max_i |x_{i+1}-x_{i}| \le h$. Then $|\hat f(x_j(\varepsilon),\varepsilon) -  \tilde f(x_j(\varepsilon),\varepsilon)| = O(h)$ for sufficiently small $h$.
\label{ODE_approx}
\end{lemma}

\begin{proof}
This lemma is a direct consequence of the well-known error estimate for Euler's method. Consider the general ODE $y' = g(y)$, $y(x_0) = y_0$. Theorem 7.5 on pp. 40 of  \cite{hairer2009solving} states that the bounds $|g| \le A$, $|g'| \le L$ in a neighborhood of the solution imply that the Euler polygons $y_h(x)$, on a grid $x_i$ with maximum spacing at most $h$, obey $|y_h(x) - y(x)| \le A(\exp(L(x-x_0))-1)h$. In lemmas \ref{ode_unique}, \ref{inexact_ode}, we have shown for the inexact start ODE  \eqref{ode} started at $\tilde v_0$ that $g,g'$ are continuous in a neighborhood of the solution, so the bounds $A,L$ exist. We only consider the finite interval $[\hat l,\hat u]$, therefore the exponential term is bounded. We get $|y_h(x) - y(x)| =O(h)$, as required.
\end{proof}

If we take $h = h(\varepsilon) \to 0$ as $\varepsilon \to 0$, then the above lemma shows that $|\hat f(x_j(\varepsilon),\varepsilon) -  \tilde f(x_j(\varepsilon),\varepsilon)| \to 0$ uniformly over all $x_j(\varepsilon)$. Comparing with Lemma \ref{grid_converge_lemma} and Lemma \ref{inexact_ode} (specifically bound \eqref{inexact_ode_bound}), all that remains to show for our main result is that the true density $f$ is well approximated by the Stieltjes-smoothed density $f(\cdot,\varepsilon)$, i.e.: $f(x_j(\varepsilon),\varepsilon)-f(x_j(\varepsilon)) \to 0$. Since the $x_j(\varepsilon)$ depend on $\varepsilon$, it is necessary to show that $f(x,\varepsilon)-f(x) \to 0$ uniformly in $x$, where recall that $f(x,\varepsilon) = \mathrm{Imag}(m(x+ i\delta(\varepsilon)))/\pi$.

\begin{lemma}{\bf Approximation of a density by its Stieltjes transform:}
Let $f$ be a bounded probability density function. Denote by $m(z)$ its Stieltjes transform: $m(z) = \int f(x)/(x - z)\, dx.$  Suppose $f$ is uniformly continuous. Then as $\varepsilon \to 0$:
$$
\sup_{x \in \mathbb{R}} \left|\frac{1}{\pi}\mathrm{Imag}(m(x+ i\varepsilon)) -f(x)\right| \to 0,
$$
\label{stieltjes_bound}
\end{lemma}

\begin{proof}
A simple calculation reveals
$$ f(x,\varepsilon) := \frac{1}{\pi}\mathrm{Imag}(m(x+ i\varepsilon)) = \frac{1}{\pi} \int \frac{\varepsilon f(t)\, dt}{(t-x)^2 + \varepsilon^2}.$$
Clearly 
\begin{equation}
1 = \frac{1}{\pi} \int \frac{\varepsilon \, dt}{(t-x)^2 + \varepsilon^2},
\label{integral_identity}
\end{equation}
 
therefore we can define $g$ such that 
$$
f(x,\varepsilon)- f(x) = \frac{1}{\pi} \int \frac{\varepsilon (f(t)-f(x))\, dt}{(t-x)^2 + \varepsilon^2} = \int g(t)\,dt.
$$

We will bound this by breaking it down into two integrals: one near $x$ and another far from $x$. Let $\eta>0$ be the parameter determining the split, to be chosen later. We bound first the integral of $g$ near $x$:
$$
 \left|\int_{|t-x|\le\eta} g(t)\,dt\right| \le \frac{1}{\pi} \int_{|t-x|\le\eta} \frac{\varepsilon \, dt}{(t-x)^2 + \varepsilon^2} \cdot \sup_{t:|t-x|\le\eta} |f(t)-f(x)| \le \omega(\eta,f)
$$

Above we used  \eqref{integral_identity} to upper bound the integral term; and we introduced $\omega$, a modulus of continuity for $f$. This gives a bound for the integral of $g$ near $x$.

Next we bound the integral away from $x$: 
$$ \left|\int_{|t-x|>\eta} g(t)\,dt\right| \le 2|f|_{\infty} \frac{1}{\pi} \int_{|t-x|>\eta} \frac{\varepsilon \, dt}{(t-x)^2 + \varepsilon^2}.$$
The integral term in the upper bound can be evaluated explicitly as $\pi - 2\arctan(\eta/\varepsilon)$, and can be bounded above by $\pi\varepsilon/\eta$.

Putting everything together, we find the bound
$$|f(x,\varepsilon) - f(x)| \le \omega(\eta,f) + \frac{2|f|_{\infty} \varepsilon}{\eta}.$$
Choosing $\eta = \varepsilon^{\alpha}$ with $\alpha \in (0,1)$ yields a bound that tends to zero uniformly over $x$, when $\varepsilon \to 0$. The modulus of continuity tends to zero because $f$ is uniformly continuous. Thus we have shown the desired claim and finished the proof.

Note that, if $f$ is continuously differentiable in the neighbrhood of a point $x$, then we obtain an optimal bound on $|f(x,\varepsilon) - f(x)|$ by taking $\eta = c \varepsilon^{1/2}$. This guarantees $|f(x,\varepsilon) - f(x)| = O(\varepsilon^{1/2})$. The square root scaling in $\varepsilon$ motivates us to work on the line with imaginary part $\delta = \varepsilon^2$ throughout the paper, to get accuracy of order $\varepsilon$.
\end{proof}

The previous results can be put together to prove Theorem \ref{density_approx}.

\begin{proof}[of Theorem \ref{density_approx}]
\label{main_proof}
  We shall show the uniform convergence of the density approximation $\hat f$. By lemma \ref{grid_converge_lemma}, we only need to show the uniform convergence on the grid $x_i(\varepsilon)$ within the support intervals. Let $[l,u]$ be such an interval, let $[\hat l,\hat u]$ be the approximation produced by \textsc{Spectrode} for some $\varepsilon$. Also $x_i(\varepsilon)$ is the uniformly spaced grid $\hat l  = x_0 < x_1 < \ldots < x_M  = \hat u$ of length $M = \lceil\varepsilon^{-1/2}\rceil $. 

Lemma \ref{ODE_approx} is applicable to this grid, because the grid width scales as $\propto \varepsilon^{1/2} \to 0$. By this lemma $\max_j|\hat f(x_j(\varepsilon),\varepsilon) -  \tilde f(x_j(\varepsilon),\varepsilon)| \to 0.$

Further, Lemma \ref{inexact_ode} (specifically bound \eqref{inexact_ode_bound}) applies if $\varepsilon$ is sufficiently small; and if for fixed $\varepsilon$, the accuracy parameter $\eta = \eta(\varepsilon)$ in the fixed-point algorithm is sufficiently small. By bound \eqref{inexact_ode_bound}, and because $c(\varepsilon) \to 0$,  we get $
\max_j|\tilde f(x_j(\varepsilon),\varepsilon) - f(x_j(\varepsilon),\varepsilon)|   \to 0 .
$

On the other hand, $f$ is continuous and compactly supported, hence uniformly continuous. Therefore, by Lemma \ref{stieltjes_bound}: 
$\sup_x|f(x,\varepsilon)-f(x)| \to 0.$
Putting everything together:
$\max_j|\hat f(x_j(\varepsilon),\varepsilon) -f(x_j(\varepsilon))| \to 0.$

This precisely fulfills the hypotheses of Lemma \ref{grid_converge_lemma}, and that lemma gives the desired claim.
\end{proof}
\subsection{Non-atomic measures}
\label{non_atomic}

Our algorithm makes sense for general non-atomic limit PSD-s. Indeed, for general PSD $H$ the ODE \eqref{ode} takes the form
\begin{equation*}
\frac{dv}{dx} = \mathcal{F}(v) := \frac{1}{\frac{1}{v^2} -\gamma\int \frac{t^2\,dH(t)}{(1 + tv)^2}}, \mbox{ } v(x_0) = v_0.
\label{ode_general}
\end{equation*}

This ODE can be implemented and solved efficiently as long as the integral in the denominator is convenient to compute. Lemma \ref{silv_choi} provides the means to find the support, and the fixed-point algorithm  converges to a starting point even at this level of generality. Therefore the ODE approach is more generally applicable than this paper's restriction to atomic measures. 

In our current implementation of \textsc{Spectrode}, we go beyond mixtures of point masses and also allow mixtures of uniform distributions, so $H = \sum_{i=1}^{J} w_i \delta_{t_i} + \sum_{t=1}^{T} w^*_t U_{a_t,b_t}$, where $U_{a,b}$ is a uniform distribution on $[a,b]$. An example was shown in Figure \ref{two_examples}. To efficiently support uniform distributions, we compute in closed form the integrals that appear in the FPA iteration in the function $h(v)$ in algorithm \ref{fp.alg}, and in the ODE. By linearity of the integral, we only need to do the calculation for the individual uniform distributions. We use the formulas:
$$\int_{-\infty}^{\infty} \frac{t\,dU_{a,b}(t)}{1 + tv} = \frac{1}{v} - \frac{\log \frac{bv+1}{av+1}}{(b-a)v^2}, \,\,\,\,\,\, 
\int_{-\infty}^{\infty} \frac{t^2\,dU_{a,b}(t)}{(1 + tv)^2} = \frac{1}{v^2} - \frac{2\log \frac{bv+1}{av+1} + \frac{1}{bv+1}-\frac{1}{av+1}}{(b-a)v^3}.$$


Armed with these, we obtain efficient computation with arbitrary finite mixtures of uniform distributions. 

In addition, a large part of the analysis holds true. Specifically, the convergence of FPA and the analysis of the ODE do not use the atomic structure directly. Instead, the atomic structure of the PSD is used through the structure of the support of the ESD $F$ as a a union of compact intervals; and the behavior of $z'$ characterizing the support (claim 3 of Lemma \ref{silv_choi}).

  To our knowledge, these claims are currently not known to hold for more general PSD. Indeed, in the very recent related work  \cite{hachem2014large} examining the fluctuations of the spectrum at the edges of the support, the authors work conditionally, assuming that the edges are regular in a certain sense. Extending the validity of our algorithm would presumably require developing a better understanding of the support of ESDs for general PSDs. This could be an interesting direction for future research. 


\section{Applications}
\label{Consequences and Extensions}

In this section we apply \textsc{Spectrode} to compute moments of the limit ESD and contour integrals of its Stieltjes transform. 

\subsection{Moments of the ESD}
\label{moments}

The uniform convergence of the approximated density allows us to compute nearly arbitrary moments of the ESD. These moments have many applications, see \cite{tulino2004random,couillet2011random,yao2015large}. 

Obtaining the moments of the ESD is in general difficult.  The polynomial moments $\mathbb{E}_FX^k$ of the ESD can be computed using free probability. However, there appear to be no general rules for calculating more general moments such as $\mathbb{E}_F\log(X)$ or $\mathbb{P}_F(X \le c)$. In contrast, they can be computed conveniently with our method.

\begin{corollary}
Let $\gamma <1$, and $h:\mathbb{R}\to \mathbb{R}$ be bounded on compact intervals and Riemann-integrable on compact intervals. Then the integral of $h$ computed against the density approximation $\hat f(\cdot,\varepsilon)$ produced by \textsc{Spectrode}  converges to the moment of $h$ under the limit ESD $F$: 

$$\lim_{\varepsilon \to 0} \int h(x) \hat f(x,\varepsilon) \, dx  =  \int h(x)  f(x) \, dx. $$

The same holds for $\gamma >1$ if we account for the point mass at $x=0$.
\label{moments_approx}
\end{corollary}

\begin{proof}
Let $M$ be an arbitrary upper bound on the support of $F$. Then for sufficiently small $\varepsilon$, $f$ is zero outside $[0,M]$, and so is $\hat f$ by Theorem \ref{density_approx}. Therefore

$$\left|\int h(x) \left(\hat f(x,\varepsilon) - f(x)\right) \, dx \right| \le \int_{0}^{M} \left| h(x)  \right| \, dx  \sup_{x\in[0,M]}|\hat f(x,\varepsilon) - f(x)| \to 0. $$
The convergence to zero follows because the first term is bounded by the assumptions on $h$, the second term tends to zero by Theorem \ref{density_approx}. The case $\gamma >1$ is completely analogous. 
\end{proof}

It is also possible to prove the convergence of Riemann sums $\sum_i  h(x_i)$ $\hat f(x_i,\varepsilon)$ $\Delta(x_i)$, but this will not be pursued here. 

As an example, in Table \ref{table:moments} we show the results of computing three moments of the standard MP distribution, with $\gamma=1/2$. The three functions are $h(x) = x$, $\log(x)$, and $\log^2(x)$. The true value of the expectation for $x$ is 1, for $\log(x)$ is $\log(2)-1$ (see e.g. \cite{yao2015large}), while for $\log^2(x)$ it is not easily available in the standard references on the subject. The numerical values computed with \textsc{Spectrode}, for a precision parameter $\varepsilon = 10^{-8}$, have a very good accuracy in the case when the true answer is available. In addition, \textsc{Spectrode} also computes the integral of $\log^2(x)$.

\begin{table}[h]
\footnotesize
\centering
\caption{Moments of standard MP law with $\gamma=1/2$.}
\begin{tabular}{|l|l|l|l|}
\hline
Function & True Value & Numerical Value & Accuracy \\ \hline
$x$	 & 1 & 1 & 2.3308e-06\\ \hline 
$\log(x)$	 & -0.30685 & -0.30684 & 1.4483e-05\\ \hline 
$\log^2(x)$	 & unknown & 0.81724 & \\ \hline 

\end{tabular}
\label{table:moments}
\end{table}

\subsection{Contour integrals of the Stieltjes transform of the ESD}
\label{contour}
\textsc{Spectrode} can be adapted to compute contour integrals involving the Stieltjes transform of the limit ESD. Such integrals appear in Bai and Silverstein's central limit theorem for linear spectral statistics of the covariance matrix \cite{bai2004clt}. 

Let $\Gamma$ be a smooth contour in the complex plane that does not intersect the support of the limit ESD $F$. Let $c:[0,1]\to\mathbb{C}$ be a parametrization of the contour, and $v(z)$ be the Stieltjes transform of $\underline F$. Note that $v(z)$ can be defined by the same formulas \eqref{ST}-\eqref{dual.ST} for all $z$ outside the support of $F$, and in particular for all $z \in \Gamma$. 

Suppose that for a smooth function $G$, we want to calculate the following integral over the clockwise oriented contour $\Gamma$:
\begin{equation}
\mathcal{I} = \oint_{\Gamma} G(z,v(z)) \, dz. 
\label{st_int}
\end{equation}

An example is the mean in the CLT for linear spectral statistics $\sum_i g(\lambda_i)$ of $\widehat{\Sigma}$ (for a smooth function $g$), which involves the formula (see \cite{bai2004clt}):

\begin{equation}
\mathcal{J}(g,H,\gamma) =-\frac{1}{2\pi i} \oint_{\Gamma} g(z) \frac{\gamma \int \frac{v(z)^3 t^2 \, dH(t)}{(1 + t v(z))^3}}{\left[1 - \gamma \int \frac{v(z)^2 t^2 \, dH(t)}{(1 + t v(z))^2}\right]^2} \, dz.
\label{clt_mean}
\end{equation}

This is a special case of our general problem \eqref{st_int}. To compute the general integral $\mathcal{I}$, we perform the change of variables $z = c(t)$, and rewrite $\mathcal{I}$ in the form $\mathcal{I} = \int_{t\in [0,1]} G\{c(t),v(c(t))\}  c'(t)\, dt$. We assume $G(\cdot)$, $c(t)$ and $c'(t)$ are conveniently computable. Then the key problem in approximating this integral is obtaining $v(c(t))$ for the entire range $t \in [0,1]$. This is where the ideas used in \textsc{Spectrode} will help.

We will obtain $h(t) := v(c(t))$ by exhibiting an ODE for it. Specifically, by using the chain rule $h'(t) = v'(c(t))c'(t) $, and recalling the ODE \eqref{ode}, which states $v'(z) = \mathcal{F}(v(z))$, we get the new ODE for $h$: $h'(t) =\mathcal{F}(h) c'(t).$

The starting point $h(0)$ (i.e. $v(c(0))$) can again be found using the fixed-point algorithm, provided the starting point of the curve, $c(0)$ has nonzero imaginary part. Indeed, this follows from the general properties of FPA if the imaginary part of $c(0)$ is positive. Moreover, the Stieltjes transform enjoys $\bar v(z) = v(\bar z)$ ($\bar z$ denotes complex conjugation), so FPA also converges for $z$ with negative imaginary part. Now, if the curve $\Gamma$ lies entirely on the real line, then a starting point be obtained using the function \eqref{ST_inverse}, which becomes the explicit inverse of the map $c \to v(c)$, as shown in \cite{silverstein1995analysis}.

The new ODE can be integrated numerically. The obtained values for $\hat h$ can be used to approximate the contour integral $\mathcal{I}$ using standard quadrature methods. 

\subsubsection{Example}

In Table \ref{table:contour} we show an example. Consider the sample covariance matrix $\widehat{{\bf\Sigma}} = n^{-1} \mathbf{X}^\top \mathbf{X}$, where $x_{ij}$ are real random variables with $\mathbb{E}x_{ij}=0$, $\mathbb{E}x_{ij}^2=1$, and $\mathbb{E}x_{ij}^4=3$. Let $\lambda_i$ be the eigenvalues of the sample covariance matrix, and let $F_{\gamma}$ be the standard MP law with index $\gamma$. Consider a sequence of such problems with $n,p \to \infty$, $\gamma_p := p/n \to \gamma$. Let $F_p$ be discrete spectral distribution of $\widehat{{\bf\Sigma}}$. For a function  $g$ that is analytic in a neighborhood of the support of $F_{\gamma}$, let $X_p(g)$ be the linear spectral statistic $X_p(g) := p\left\{ F_p(g) - F_{\gamma_p}(g)\right\} = \sum_{i=1}^p g(\lambda_i) - p\,\mathbb{E}_{F_{\gamma_p}}\left[g(\lambda)\right].$

Then Bai and Silverstein \cite{bai2004clt} proved that $X_p(g)$ is asymptotically normal with mean given in \eqref{clt_mean}, with $H = \delta_1$ and $\Gamma$ an arbitrary contour enclosing the support of the ESD. It is known (see for instance \cite{yao2015large}) that:
\begin{align*}
\mathcal{J}(x,\delta_1,\gamma), \,\,\,\,\,\, 
\mathcal{J}(\log x,\delta_1,\gamma)  = \frac{1}{2} \log(1-\gamma).
\end{align*}

We use our method, outlined above, to compute the integral \eqref{clt_mean}. We take $\gamma=1/2$ and compare against the closed form solutions for $g(x) =x$ and $g(x)=\log(x)$. We also compute the integral for $g(x) = \log^2(x)$, for which no closed form solution appears to be known. The results, displayed below in Table \ref{table:contour}, show the excellent performance of our method. 

\begin{table}[h]
\footnotesize
\centering
\caption{Mean of normalized LSS for identity covariance.}
\begin{tabular}{|l|l|l|l|}
\hline
Linear Statistic & True Value& Numerical Value & Accuracy \\ \hline
$x$	 & 0 & -7.1292e-12 & 7.1292e-12\\ \hline 
$\log(x)$	 & -0.34657 & -0.34655 & 2.2214e-05\\ \hline 
$\log^2(x)$	 & unknown  & 1.2111 & \\ \hline 
\end{tabular}
\label{table:contour}
\end{table}

In this experiment, we used the circle contour $c(t) = a/2 + a/2\cdot e^{2\pi i t}$, with $a = 1.1\cdot(1+\gamma^{1/2})^2$. This contour encloses the support of the ESD. The starting point of the ODE, $v(c(0))=v(a)$, was found using the function $z(v)$ \eqref{ST_inverse}. More specifically, using the default bisection method in \textsc{Matlab}, we found numerically the  unique solution to the equation $z(v)=a$ on the interval $v\in(-1,0)$ such that $z'(v)>0$. As discussed in Section \ref{themethod}, this guarantees the correctness of the starting point.

\section{Related Work}
\label{relatedwork}

Problems related to computing the ESD of covariance matrices have been discussed in several important works. Here we examine the strengths and weaknesses of related and alternative methods.

\subsection{Monte Carlo}

Monte Carlo (MC) simulation can be used to approximate the eigenvalue density of large covariance matrices via a smoothed empirical histogram of eigenvalues. It was proved in \cite{jing2010nonparametric} that this method consistently estimates the ESD. However, we show in a simple simulation that MC is prohibitively slow when more than three digits of accuracy is required. 

\begin{figure}
  \centering
  \includegraphics[scale=0.35]
  {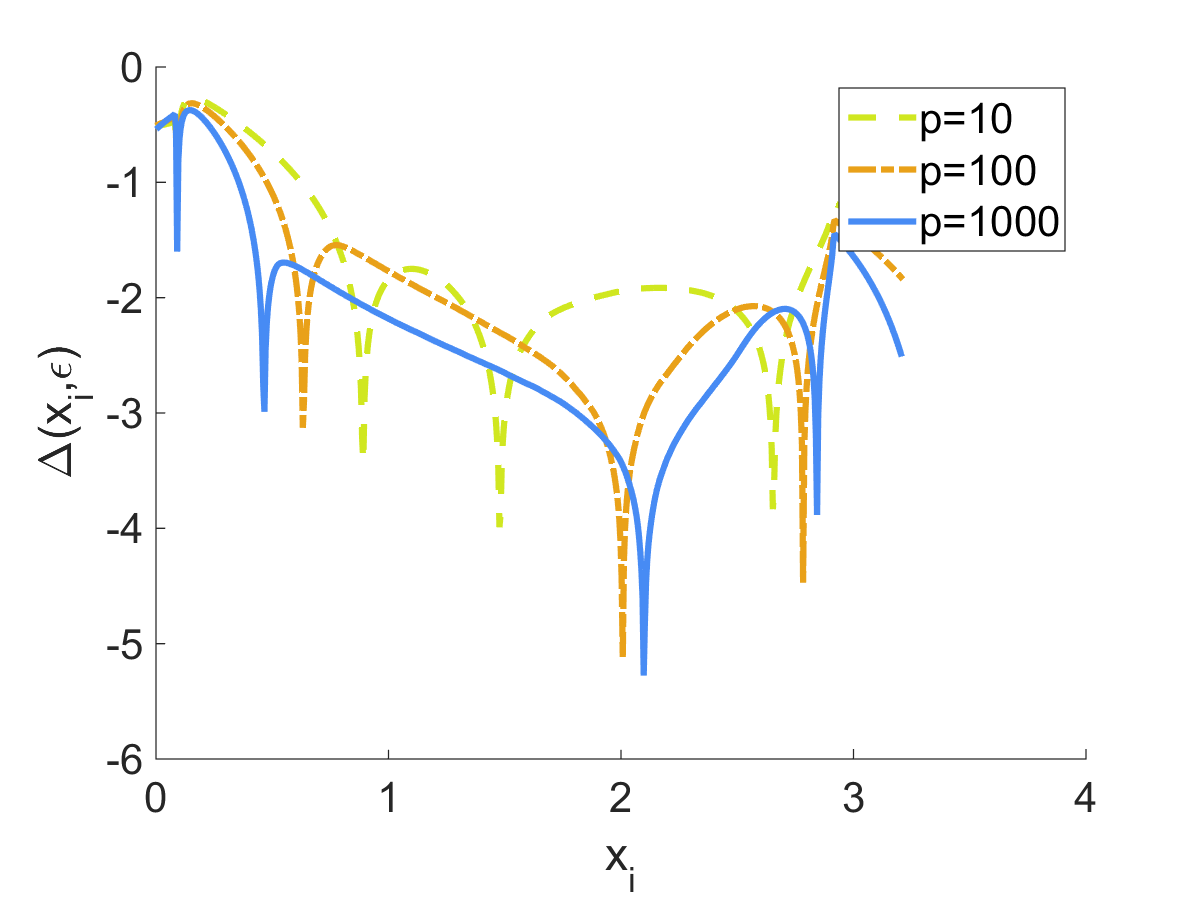}

\caption{Computational accuracy of the Monte Carlo method, discussed in Section \ref{relatedwork}.  Here we sample $n_\mathrm{MC} = 1000$ independent random matrices with iid Gaussian entries and aspect ratio $\gamma = 1/2$. We fit a kernel density estimator to the histogram of eigenvalues of each sample, and average over independent samples.  We display the pointwise error of approximation for $p$ in the range ten, $10^2, 10^3$}
\label{MC_accuracy}
\end{figure}

\subsubsection{Experiment setup and parameters} 

We use the {\tt MP} test problem when the covariance matrix ${\bf\Sigma}_p = I_p$ and $H=\delta_1$. For a specified pair $n,p$ we sample random matrices $\mathbf{X}$ with iid standard normal entries $\mathcal{N}(0,1)$. We compute the empirical eigenvalues of the sample covariance matrix $\widehat{{\bf\Sigma}} = \mathbf{X}^\top\mathbf{X}/n$. We fit a kernel density estimate to the histogram of eigenvalues, using an Epanechnikov kernel with automatically chosen kernel width in \textsc{Matlab}. Finally, the kernel density estimate is averaged over the $n_{\mathrm{MC}}$ independent Monte Carlo trials to get a final estimate $\hat f_{\mathrm{MC}}(x_i)$ of the density. We use the following parameters: $n_{\mathrm{MC}} = 1000$, $\gamma = 1/2$, and $p$ takes the values $10, 10^2, 10^3$.

We compare against the  the true limit $f$ from Eq. \eqref{MP_density} and report the error in the density from Eq. \eqref{error_in_density}:  $\Delta_{\mathrm{MC}}(x_i) = \log_{10}|\hat f_{\mathrm{MC}}(x_{i}) - f(x_{i})|$. 

\subsubsection{Results}
In the results in Figure \ref{MC_accuracy}, we see that number of correct significant digits is two or three. We get only 1/2 extra digit when we move from $p=100$ to $p=1000$.  The experiment for $p=1000$ takes about ten minutes on the same hardware described in Section \ref{timing}. Computing the eigenvalues using the SVD takes $\Theta(p^3)$ steps. For ten times larger $p = 10^4$, such an experiment would take about $10^5$ minutes, or cca. 1600 hours, which is prohibitively slow.

Based on this experiment, we conclude that the MC simulation for computing the ESD, if implemented in a straightforward way, is not suitable for getting more than three digits of accuracy.

\subsection{Method of Successive Approximation}

While Marchenko and Pastur \cite{marchenko1967distribution} do not place emphasis on numerically solving the Marchenko-Pastur equation, they mention that its solution can be found by the method of successive approximation (SA). SA was also proposed by Girko \cite{girko2001theory} for several variants of the Marchenko-Pastur equation. We will argue that SA is not an efficient method for computing the ESD.

Marchenko and Pastur in \cite{marchenko1967distribution} consider a more general model than this paper, allowing for an additive term ${\bf Y} ={\bf A} + n^{-1}{\bf X^\top T X}$. They denote by $\tau(\xi):[0,1] \to \mathbb{R}$ the quantile function of the PSD $H$, which is the limit of the spectrum of $\bf T$; and $c = \lim p/n$, which was called $\gamma$ in our paper. Also, they call $m_0(z)$ the Stieltjes transform of the limit spectrum of $\bf A$, writing the equation for the Stieltjes transform of the ESD of $\bf Y$ in the form (see equation (1.13) in the English translation of their paper):

\begin{equation}
u(z,t) = m_0\left(z - c\int_0^t \frac{\tau(\xi) \, d \xi}{1 + \tau(\xi) u(z,t)}\right).
\label{MP_original}
\end{equation}

The unknown function is $u(z,t)$. They show that, with the initial condition $u(z,0) = m_H(z)$, there is a unique solution $u(z,t)$ analytic in $z \in \mathbb{C}^+$ and continuous in $t \in [0,1]$ of \eqref{MP_original}. Then, $u(z,1)$ is the Stieltjes transform $m_F(z)$ of the limit ESD $F$. 

In our special case $\bf A = 0$, so $m_0(z)=-1/z$ and the equation simplifies to
$u(z,t) = -1/(z - c\int_0^t \frac{\tau(\xi) \, d \xi}{1 + \tau(\xi) u(z,t)})$.

 Denoting $v(z):=u(z,1)$, and switching from the integral over $t\in [0,1]$ to the integral against $dH$, we see that this reduces to the Silverstein equation \eqref{silv.eq}.

The method of successive approximation starts with an arbitrary function $u_0(z,t)$ obeying the smoothness and continuity conditions above. It defines the sequence of functions $u_n(z,t)$ inductively for $n \in \mathbb{N}$ by:

\begin{equation}
u_{n+1}(z,t) = m_0\left(z - c\int_0^t \frac{\tau(\xi) \, d \xi}{1 + \tau(\xi) u_n(z,t)}\right).
\label{MP_iter}
\end{equation}

For this method it is crucial that one maintains bivariate functions $u_n(z,t)$. The definitition of $u_{n+1}$ depends on the integral of $u_n$ against $t$, even if we are interested only in the values $u_{n+1}(z,1)$, for $t=1$. Therefore, the method of successive approximation relies on an additional time dimension. However, this extra dimension seems costly in the special case when $\bf A =0$. There are more efficient methods, such as FPA and \textsc{Spectrode}, that do not rely on additional dimensions.

\subsection{Fixed-Point Algorithm (FPA)}

FPA appears to be the standard approach that most researchers use to compute the ESD or sample covariance matrices in the Marchenko-Pastur asymptotic regime. Versions of FPA have been developed in many different areas, including free probability, wireless communications, signal processing, and mathematical statistics. The existing techniques for analyzing it fall into the following categories: (1) complex analytic method; (2) contraction by deterministic equivalents of random matrices; and (3) interference functions. 

Balinschi and Bercovici in \cite{belinschi2007new} explain subordination results in free probability, with the Marchenko-Pastur-Silverstein equation as a special case. They employ a  complex analytic approach, which involves the study of Denjoy-Wolff points. As a special case of their results, it follows that FPA converges. This was not explicitly stated by the authors, but it is indeed an immediate consequence of their results. 

In many random matrix models, the fixed-point algorithm is often invoked implicitly, to show the uniqueness of the solution to the fixed-point equation governing the spectrum. For instance, this is the approach in the analysis of deterministic equivalents for random matrices in \cite{hachem2007deterministic}. Here the authors show uniqueness of solutions to their equations by exhibiting bounds on the size of successive iterates, which is equivalent to the convergence of the fixed point algorithm in that context. 

In the wireless communications literature, explicit fixed-point algorithms have been developed for computing the ESDs of very general random matrix models in several papers: \cite{couillet2011deterministic, couillet2011random, couillet2012random}. The authors in \cite{couillet2011deterministic} give a fixed point algorithm for a random matrix model that is a sum of arbitrary covariance matrices. They prove its convergence by showing that the iteration is a contraction for complex points $z$ with large $\text{Imag}(z)$, similarly to the earlier work \cite{hachem2007deterministic}.  Then the convergence extends to all points outside the support of the ESD using Vitali's theorem.  

In contrast, the authors in \cite{couillet2012random} take a different approach to proving the convergence of FPA. They show that for negative arguments $z \in (-\infty,0]$, their equations are fixed points of interference functions, and they rely on general convergence results of such functions \cite{yates1995framework}.  Again, Vitali's theorem extends the convergence to other points. In the model of \cite{couillet2012random} the convergence of FPA cannot be established for all complex numbers; which is a counterexample showing that FPA is not expected to converge in all circumstances. However, the interference function approach for proving convergence of FPA is powerful and general, see for instance \cite{Couillet2015random, morales2015large}.

FPA has appeared in other papers as well. Yao in \cite{yao2012note} has a fixed point algorithm for a time series problem, but without convergence guarantees. Hendrikse  and his coauthors \cite{hendrikse2013smooth} discover the fixed-point algorithm for the limit ESD of sample covariance matrices, and claim to prove convergence for the argument $z$ with sufficiently large imaginary part $\text{Imag}(z)>c$, using elementary arguments. However, their arguments appear to be incomplete; as they do not show that the iterates $z_t$ remain in the region $\text{Imag}(z)>c$ for all $t$. 

In a free probability context, Belinschi, Mai and Speicher in the paper \cite{belinschi2013analytic} generalize the fixed-point algorithm to compute the limit ESD of arbitrary polynomials  $P(X_{n1},...,X_{nk})$, given the limit ESD of random matrices $X_{n1},..,X_{nk}$. This important paper has the advantage of generality, and in principle can subsume other methods. As we showed, however, for our problem FPA can unfortunately be slow for high-precision computations.

\subsection{Other Methods}

Special cases of limit ESDs have been computed in a case-by-case fashion. For instance, in the analysis of wireless networks, \cite{Morgenshtern2007Crystallization} develop a computational procedure for a special case that involves a fourth-order polynomial equation. 

Further, Nadakuditi and Edelman in the influential work \cite{rao2008polynomial} developed a polynomial method as a general framework for computing limit spectra of ensembles whose Stieltjes transform is algebraic. As noted by the authors (see their Section 7), these methods in general do not lead to an automated way to compute the limit density. Indeed, this is presented as an open problem in \cite{rao2008polynomial}. \textsc{Spectrode} addresses a narrower setting and shows that the ESD can be computed reliably in that setting.

Olver and Nadakuditi in \cite{olver2012numerical} present an interesting approach for calculating the additive, multiplicative and compressive convolution in free probability. The map from population to sample spectra $H \rightarrow F$ corresponds to free multiplicative convolution with the identity Marchenko-Pastur distribution. However, this method is not generally applicable to our problem. It requires that the support of the LSD $F$ be precisely one compact interval, because it relies on specific series expansions (see their Section 4). In our case of this is often not the case. We see \textsc{Spectrode} as complementary to their method, for the case of multiple intervals in the support.

\section{Software}
\label{software}

A software companion for this paper is available at \url{https://github.com/dobriban/eigenedge}. It contains implementations of

\begin{enumerate}
\item methods to compute the sample spectrum: \textsc{Spectrode} and the fixed-point method
\item methods to compute arbitrary moments and quantiles of the ESD
\item \textsc{Matlab} scripts to reproduce all computational results of this paper
\item detailed documentation with examples 
\end{enumerate}

The package is  user-friendly. Once the appropriate environment is installed, the ESD of a uniform mixture of four point masses at \verb+t = [1; 2; 3; 4]+, and with aspect ratio \verb+gamma = 1/2+ requires three lines of code:

\begin{verbatim}
t =  [1; 2; 3; 4]; 
gamma = 1/2; 
[grid, density] =  spectrode(t, gamma); %compute limit ESD
\end{verbatim}

%

\section*{Acknowledgements}
We are grateful to David L. Donoho for posing the problem and reviewing the manuscript. We are obliged to Romain Couillet for numerous references and comments, especially on FPA.  We thank Iain M. Johnstone, Matthew McKay, Art B. Owen and Jack W. Silverstein for helpful comments; and Tobias Mai for discussions about the paper \cite{belinschi2013analytic}. Financial support has been provided by NSF DMS 1418362.

\printbibliography

\end{document}